\documentclass[11pt]{amsart}
\usepackage{amssymb,amsmath,amsfonts,amscd,euscript}
\usepackage[matrix,arrow,curve]{xy}

\newcommand{\nc}{\newcommand}

\numberwithin{equation}{section}
\newtheorem{thm}{Theorem}[section]
\newtheorem{prop}[thm]{Proposition}
\newtheorem{lem}[thm]{Lemma}
\newtheorem{cor}[thm]{Corollary}
\theoremstyle{remark}
\newtheorem{rem}[thm]{Remark}
\newtheorem{definition}[thm]{Definition}

\nc{\gl}{\mathfrak{gl}}
\nc{\GL}{\mathfrak{GL}}
\nc{\g}{\mathfrak{g}}
\nc{\gh}{\widehat\g}
\nc{\h}{\mathfrak{h}}
\nc{\la}{\lambda}
\nc{\al}{\alpha }
\nc{\be}{\beta }
\nc{\ve}{\varepsilon }
\nc{\om}{\omega }

\nc{\ta}{\theta}
\nc{\ch}{{\mathop {\rm ch}}}
\nc{\Tr}{{\mathop {\rm Tr}\,}}
\nc{\Id}{{\mathop {\rm Id}}}
\nc{\ad}{{\mathop {\rm ad}}}
\nc{\bra}{\langle}
\nc{\ket}{\rangle}
\nc{\x}{{\bf x}}
\nc{\bm}{{\bf m}}
\nc{\bs}{{\bf s}}
\nc{\ba}{{\bf a}}
\nc{\bb}{{\bf b}}
\nc{\bk}{{\bf k}}
\nc{\bp}{{\bf p}}
\nc{\pa}{\partial}
\nc{\ld}{\ldots}
\nc{\cd}{\cdots}
\nc{\hk}{\hookrightarrow}
\nc{\T}{\otimes}
\nc{\gr}{\mathrm{gr}}
\nc{\ov}{\overline}

\nc{\cO}{\mathcal O}
\nc{\msl}{\mathfrak{sl}}
\nc{\mgl}{\mathfrak{gl}}
\nc{\U}{\mathrm U}
\nc{\V}{\EuScript V}
\nc{\cL}{\mathcal{L}}
\nc{\Res}{\mathrm{Res\ }}

\newcommand{\bZ}{{\mathbb Z}}

\newcommand{\fh}{{\mathfrak h}}

\newcommand{\fg}{{\mathfrak g}}

\newcommand{\fn}{{\mathfrak n}}

\begin{document}

\title[Generalized Weyl modules for twisted current algebras]
{Generalized Weyl modules for twisted current algebras}

\author{Evgeny Feigin}
\address{Evgeny Feigin:\newline
Department of Mathematics,\newline
National Research University Higher School of Economics,\newline
Vavilova str. 7, 117312, Moscow, Russia,\newline
{\it and }\newline
Tamm Theory Division, Lebedev Physics Institute
}
\email{evgfeig@gmail.com}
\author{Ievgen Makedonskyi}
\address{Ievgen Makedonskyi:\newline
Department of Mathematics,\newline
National Research University Higher School of Economics,\newline
Vavilova str. 7, 117312, Moscow, Russia
}
\email{makedonskii\_e@mail.ru}

\begin{abstract}
We introduce the notion of generalized Weyl modules for twisted current algebras.
We study their representation-theoretic and combinatorial properties and connection to the theory of nonsymmetric Macdonald polynomials.
As an application we compute the dimension of the classical Weyl modules in the remaining unknown case.
\end{abstract}

\maketitle

\section*{Introduction}

Let $\fg^{af}$ be an affine Lie algebra. It has a natural $\mathbb{Z}$-grading
$\fg^{af}=\bigoplus_{k=-\infty}^{\infty}\fg(k)$; in particular, $\fg(0)=\fg_0 \oplus \mathbb{C} K \oplus \mathbb{C} d$,
where $\fg_0$ is a finite-dimensional simple Lie algebra, $K$ is central and $d$ is the degree element.
Consider the subalgebra $\fg^{af}_{\geq 0}=\fg_0 \oplus \bigoplus_{k=1}^{\infty}\fg(k)$ ($\fg^{af}_{\geq 0}$
coincides with a maximal parabolic subalgebra modulo $K$ and $d$).
If $\fg^{af}$ is untwisted, then $\fg^{af}_{\geq 0}$ is the Lie algebra of polynomial currents:
$\fg^{af}_{\geq 0} \simeq \fg_0 \otimes \mathbb{C}[t]$.
If $\fg^{af}$ is twisted, then  $\fg^{af}_{\geq 0}$ is called twisted current algebra.
Let $\Delta$ be the root system of $\fg_0$, $\Delta_+$ ($\Delta_-$) be the set of positive (negative) roots and
$\delta$ be the basic imaginary root. We fix a Cartan decomposition $\fg_0=\fn_- \oplus \fh \oplus \fn_+$,
where $\fh$ is the Cartan subalgebra, $\fn_+$ ($\fn_-$) is the linear span of all positive root vectors
(negative root vectors).
Then all real roots of $\fg^{af}$ are of the form $l\alpha + k \delta$,
$\alpha \in \Delta$, $k \in \mathbb{Z}$, $l \in \{1,2\}$ (note that $l=1$ unless $\fg^{af}$ is of type $A_{2n}^{(2)}$).
We fix a vector $e_{\alpha}$ from each real $\alpha$-root space.
We denote by $\fh_{k \delta}\subset \fg^{af}$ the (imaginary) root space corresponding to the root $k \delta$.
The classical Weyl modules are labeled by dominant weights $\lambda$ of $\fg_0$ (see \cite{CIK,CFS,FK}).
The module $W(\lambda)$ is defined as the cyclic $\fg^{af}_{\geq 0}$-module with a generator $v$ and the following set of relations:
\begin{gather*}
\fh_{k\delta}v=0, k \geq 1;\ h v=\la(h)v\text{ for all } h\in\fh;\\
e_{\alpha+k\delta} v=0, k \geq 0;\ e_{-\alpha}^{\langle \alpha^\vee, \lambda \rangle+1}v=0
\text{ for all } \al\in\Delta_+.
\end{gather*}

These modules for untwisted affine Lie algebras were extensively studied (see \cite{CP,CL,FL1,FL2}). For simply-laced algebras
Weyl modules are isomorphic to level one Demazure modules.
Weyl modules for twisted Lie algebras were introduced in \cite{CFS}.

It was proven in \cite{S,I,CI} that the characters of Weyl modules coincide with the specializations of 
nonsymmetric Macdonald polynomials at $t=0$.
We introduce the following generalization of Weyl modules. The generalized Weyl modules are modules over the algebra $\fn^{af}$ spanned by
all root vectors $e_{\beta}$ corresponding to positive roots $\beta$ of $\fg^{af}$.
Assume that $\fg$ is of dual untwisted type and let $\Delta^s$ ($\Delta^l$) be the set of short (long) roots of $\Delta$.
We define ${\rm len}(\alpha)=1$ for $\alpha \in\Delta^s$ and  ${\rm len}(\alpha)=j$ for $\alpha \in\Delta^l$, where $j$ is the ratio of
squared lengths of long and short roots.
For an element $\sigma$ in the Weyl group $W$ of $\fg_0$  and $\alpha \in \Delta_+$ we
define the action of $\sigma$ on the set $\Delta_+\cup(\Delta_-^s+\delta)\cup(\Delta_-^l+j\delta)$ as follows:
\begin{gather*}
\widehat \sigma(\alpha)=   \begin{cases}
                             {\sigma(\alpha)}, & \text{ if } \sigma(\alpha) \in \Delta_+ \\
                             {\sigma(\alpha)}+ {\rm len}(\alpha)\delta, & \text{ if } \sigma(\alpha) \in \Delta_-.
                            \end{cases},\\
\widehat \sigma(-\alpha+{\rm len}(\alpha)\delta)=   \begin{cases}
                             {\sigma(-\alpha)+{\rm len}(\alpha)\delta}, & \text{ if } \sigma(-\alpha) \in \Delta_- \\
                             {\sigma(-\alpha)}, & \text{ if } \sigma(-\alpha) \in \Delta_+ \\
                            \end{cases}.									
\end{gather*}

For $\beta \in \Delta_+\cup(\Delta_-^s+\delta)\cup(\Delta_-^l+j\delta)$ we define $\widehat \sigma(e_{\beta})=e_{\widehat \sigma(\beta)}$.

\begin{definition}
Let $\fg^{af}$ is not of type $A^{(2)}_{2n}$. Then 
for an arbitrary $\fg_0$-weight $\mu$ let us write $\mu=\sigma(\lambda)$, $\sigma\in W$, $\lambda$ is antidominant. Then the (generalized) Weyl module
$W_{\mu}$
is the cyclic $\fn^{af}$ module with a generator $v$ and the following relations:
\begin{gather*}
\fh_{k\delta}v=0, k \geq 1;\\
\widehat{\sigma}(e_{-\alpha+{\rm len}(\alpha) \delta}) v=0; \
(\widehat{\sigma}(e_{\alpha}))^{-\langle \alpha^\vee, \lambda \rangle+1} v=0,\ \alpha \in \Delta_+.
\end{gather*}
\end{definition}

We also define the generalized Weyl modules in types $A_{2n}^{(2)}$
(see Definition \ref{weylmodulesmix}).
In this there exist two definitions of the classical Weyl modules, corresponding to two natural choices of the maximal
parabolic subalgebra (see \cite{CFS,CIK,FK}).
The family of the generalized Weyl modules include both types of classical Weyl modules.

Dimensions of classical Weyl modules were computed in \cite{CL,N,FK, CIK}. In these papers
the dimensions were computed  using a link to the theory of Demazure modules.
The only remaining unknown case was $\fg^{af}$ of type $A_{2n}^{(2)}$ and $\langle \lambda,\alpha_n \rangle \in 2 \mathbb{Z}$.
Using our approach we compute the dimensions in this remaining case.

The properties of generalized Weyl modules are closely related to the theory of nonsymmetric Macdonald polynomials
$E_{\lambda}(x,q,t)$ \cite{Ch1,Ch2,OS}. More precisely, we prove the following Theorem (see \cite{FM3} for
the untwisted case).

\medskip

\noindent{\bf Theorem A}. Let $\la$ be an anti-dominant $\fg_0$-weight, $\sigma\in W$, $\fg^{af}$ is not of type
$A_{2n}^{(2)}$.
Then:
\begin{enumerate}
\item $\dim W_{\sigma(\la)}=\dim W_\la$.
\item ${\rm ch} W_{w_0\la}=w_0E_{\la}(x,q^{-1},\infty)$.
\item For any $i=1,\dots,{\rm rk}(\fg)$ such that $\langle\la,\al_i^\vee\rangle < 0$ the module $W_{\sigma(\la)}$
can be decomposed into subquotients of the form $W_{\kappa(\la+\omega_i)}$, $\kappa\in W$. The subquotients are labeled
by certain alcove paths and the number of subquotients is equal to the dimension of the fundamental classical
Weyl module $W(\om_i)$.
\end{enumerate}

\medskip

Theorem A with minor modifications holds in types $A_{2n}^{(2)}$ as well. In particular, we prove the following
new dimension formula for the classical Weyl modules:
\begin{equation}\label{A2dim}
\dim W(\sum_{i=1}^{n-1} m_i\om_i+2m_n\om_n)=\prod_{i=1}^n \binom{2n+1}{i}^{m_i}.
\end{equation}

We note that all the Weyl modules we consider in the paper are local. However it is straightforward to define
the global generalized Weyl modules over twisted current algebras analogous to the untwisted case \cite{FMO}.
These modules can be studied in same way as in \cite{FMO} using the results of this paper.

The paper is organized in the following way. In Section \ref{OSformula} we recollect all the needed material
on alcove paths \cite{RY,GL} and combinatorial formula of Orr and Shimozono \cite{OS}. In Section \ref{GenWeylMod} we define
the class of generalized Weyl modules in the twisted settings and study their properties. The proof of
the main theorem is reduced to the case of rank one and rank two level zero algebras $\fg_0$.
In Section \ref{LRC} we work out the case of low rank twisted algebras.

\section{Orr-Shimozono formula}\label{OSformula}

\subsection{Twisted Quantum Bruhat Graph}
Let $\Delta=\Delta_+ \sqcup \Delta_-$ be the root system and $X$ be the weight lattice of a finite-dimensional simple nonsimply-laced Lie algebra $\fg_0$,
$W=W(X)$ be its Weyl group. We denote by $\al_i$ and $\om_i$, $i=1,\dots,{\rm rk}(\fg_0)$ simple roots and fundamental weights
of $\fg_0$.  We write $X=X_+\sqcup X_-$, where $X_+$ is the positive cone spanned by $\om_i$ and $X_-=-X_+$.  Let $s_1, \dots, s_n$, $n={\rm rk}(\fg_0)$
be the set of simple reflections in $W$. For a root $\alpha$ we denote by $s_{\alpha}$ the reflection at this root.
For $w \in W$ let $l(w)$ be the length of the element $w$.
The twisted quantum Bruhat graph \cite{Lu,BFP} (tQBG for short) is the labeled graph whose set of vertices is $W$ with edges
$w \stackrel{\alpha}{\longrightarrow} w s_{\alpha}$ such that
\begin{itemize}
\item $l(ws_{\alpha})=l(w)+1$: an edge of the usual Bruhat graph;
\item $l(ws_{\alpha})=l(w)-\langle 2\rho^{\vee},\alpha \rangle+1$: a {\it quantum edge}.
\end{itemize}
Here $2\rho^\vee=\sum_{\alpha\in \Delta_+}\alpha^\vee$, the sum of positive coroots. Note that this is the standard quantum Bruhat graph for the dual
Lie algebra $\fg_0^\vee$. The difference between the QBG and tQBG is the interchange of roots and coroots in the second condition (the QBG version contains
$\langle 2\rho,\alpha^{\vee}\rangle$).

\subsection{Alcove paths}\label{alcovepaths}
Our first goal is to define the notion of the twisted quantum alcove path.
Let $W^a$ and $W^e$ be the affine and the extended affine Weyl groups for $\fg_0^\vee$. In particular, $W^e=W^a\rtimes\Pi$,
$\Pi=W^e/W^a$ and
we have the natural embedding $X\subset W^e$ (see \cite{OS,FM3}). For $\la\in X$ we denote by $t_\la$ the corresponding element in $W^e$.
Each element $w\in W^e$ defines the set of affine roots $\beta_1(w),\dots,\beta_l(w)$, $l=l(w)$.
Given a starting point $u\in W^e$, each subset $J\subset\{1,\dots,l\}$ produces an alcove walk $p_J$, which is a sequence of elements
$(\sigma_0,\dots,\sigma_l)$ of $W^e$, $\sigma_0=u$.
Now a walk $p_J$ defines an alcove path $(z_0,\dots,z_r)$, see e.g. \cite{FMO}, formula $(2.3)$.
In particular, we denote $z_0=uw$ and $z_r=:{\rm end}p_J$.
We denote the resulting path by the same symbol $p_J$.

Recall that $W^a$ is the semi-direct product of the finite Weyl group $W$ and the lattice $X$.  For $w\in W^e$ we write
$w=\pi t_{{\rm wt} (w)}{\rm dir} (w)$, where $\pi\in \Pi$, ${\rm dir} (w)\in W$.
We thus have a natural projection ${\rm dir}:W^e\to W$.

Let us define the set $t\mathcal{QB}(u,w)$ of the  twisted quantum alcove paths.
We say that a path $p_J$ is a twisted quantum alcove path  if the projection
${\rm dir}(p_J)$ is a path in the twisted quantum Bruhat graph of $W$.

Now let us define the set $t\mathcal{QB}(u,w)$ for types $A^{(2)}_{2n}$ and ${A^{(2)}_{2n}}^\dag$.
We take $W$ to be the Weyl group of type $B_n$ and the tQBG of type $D_{n+1}^{(2)}$.
Then we say that $p_j$ is a twisted quantum alcove path if its projection onto the tQBG of type $D^{(2)}_{n+1}$
is the path in this graph and for  $A^{(2)}_{2n}$ (${A^{(2)}_{2n}}^\dag$) there is no quantum  (Bruhat)
edges coming from $\beta_i$'s with odd imaginary part (see \cite{OS}).

We say that a path $p_J \in t\widetilde{\mathcal{QB}}(u,w)$ if the projection to tQBG is a path
in the reversed quantum Bruhat graph of $W$.
Let $J^-\subset J$ be the set of $j_m\in J$ such that the root ${\rm Re}(z_m\beta_{j_m})$ is negative
(here ${\rm Re}$ denotes the real part of an affine root).
We note that $j\in J^-$ if and only if the corresponding edge in the quantum Bruhat graph is quantum.

Let $\delta$ be the basic imaginary root. For any element of the affine root lattice $\mu + N\delta$,
where $\mu\in X$  we denote $\deg(\mu + N\delta)=N$.
For an alcove path $p_J$ we define ${\rm qwt}(p_J)=\sum_{j\in J^-} \beta_j$.

\begin{definition} \label{Cdef}
For any $u,w \in W^e$ we define:
\[
C_u^{w}(x,q)=\sum_{p_J \in t\mathcal{QB}(u,w)}x^{wt(end(p_J))}q^{\deg({\rm qwt}(p_J))}.
\]
\end{definition}


In the rest of this section we describe the properties of the function $C_u^w$.
One has
$C_{t_\mu u}^{w}=x^\mu C_u^{w}.$

\begin{thm}
\cite{OS}\label{specializationOS0}
Let $\lambda \in X$ be an antidominant weight.
Then
\begin{enumerate}
\item $E_{\lambda}(x; q, 0)=C_{\rm id}^{t_\la}$.
\item $E_{\lambda}(x; q^{-1}, \infty)=\sum_{p_J\in t\widetilde{\mathcal{QB}}(\lambda)}x^{wt(end(p_J))}q^{\deg({\rm qwt}(p_J))}$,
\item $E_{\lambda}(x; q^{-1}, \infty)=w_0 C_{w_0}^{t_\la}.$
\end{enumerate}
\end{thm}

Let $r_i=l(t_{-\om_i})$  and let $\beta_j^i=\beta_j(t_{-\om_i})$, $j=1,\dots,r_i$.

\begin{lem}\label{firstbetas}
For any $\la\in X_-$ the sequence of roots $\beta_j(t_{\lambda-\omega_i})$ is equal to
\[
\beta_1^i + \langle {\beta_1^i}^\vee,\lambda \rangle \delta, \dots, \beta_{r_i}^i + \langle {\beta_{r_i}^i}^\vee,\lambda \rangle \delta,\
\beta_1,\dots,\beta_a,
\]
where $\beta_1,\dots,\beta_a$ is the sequence of $\beta$'s for $t_\la$.
\end{lem}

Recall that for an affine root $\beta$ we write $\beta={\rm Re}(\beta) +\deg (\beta)\delta$.
For $\gamma\in\Delta$ we denote by  ${\rm len}(\gamma)$ the normalized length of $\gamma$, so
${\rm len}(\gamma)=1,2,3$ ($1$ for short roots $\gamma$).

\begin{prop}\label{descriptionbeta}
$a).$\ For any reduced decomposition of $t_{-\omega_i}$ the roots $\beta_j(t_{-\om_i})$ satisfy the following properties:
\begin{itemize}
\item $\lbrace {\rm Re} \beta_j^i \rbrace=\lbrace \gamma \in \Delta_-|\langle \gamma^\vee, \omega_i \rangle<0 \rbrace$,
\item $|\lbrace j|{\rm Re}\beta_j^i=\gamma \rbrace|=-\langle \gamma^\vee, \omega_i \rangle$,
\item For any $\gamma$ the set $\lbrace \beta_j|{\rm Re}\beta_j^i=\gamma \rbrace$ is equal to
$\lbrace \gamma+{\rm len}(\gamma)\delta,  \dots, \gamma - {\rm len}(\gamma)\langle \gamma^\vee, \omega_i \rangle \delta\rbrace$.
\end{itemize}

$b).$\ There exists a reduced decomposition of $t_{-\omega_i}$ giving the following order on $\beta$'s.
We set $i_1=i$, and let $i_k$, $k=2, \dots, n$, be some ordering of the set $\lbrace 1,\dots, n \rbrace \backslash \lbrace i \rbrace$.
Let us write $(\beta_j^i)^{\vee}=-a_{i_1}\alpha_{i_1}^{\vee} - \dots - a_{i_n}\alpha_{i_n}^{\vee}+ Dd$.
Then the order on $\beta$'s is given
by the lexicographic order on the vectors $({\rm len}({\rm Re}\beta_j^i) \frac{a_{i_1}}{D},\frac{a_{i_2}}{a_{i_1}}, \dots, \frac{a_{i_n}}{a_{i_1}})$.
\end{prop}

\begin{proof}
The proof is completely analogous to the proof of Proposition 1.15 in \cite{FM3}.
\end{proof}

\begin{cor}\label{order}
$i)$ $\beta_1^i=-\alpha_i+{\rm len}(\al_i)\delta$,\\
$ii)$ if $\gamma=\tau + \eta$, $\tau, \eta \in \Delta_+^{\vee}$, $({\rm Re}\beta_j^i)^{\vee}=-\gamma$, then
\begin{multline*}
|\lbrace k|({\rm Re} \beta_k^i)^{\vee} =-\gamma, k \leq j\rbrace|=\\
|\lbrace k|({\rm Re} \beta_k^i)^{\vee} =-\tau, k \leq j\rbrace|+
|\lbrace k|({\rm Re} \beta_k^i)^{\vee} =-\eta, k \leq j\rbrace|.
\end{multline*}
$iii)$ Let $\tau, \eta\in \Delta_+^\vee$ be roots such that $\tau^\vee+2\eta^\vee\in \Delta_+^\vee$. Consider a
subsequence $\beta^i_{j_k}, k=1,\dots,p$ consisting of all roots with the property
$-({\rm Re}\beta^i_{j_k})^\vee\in \lbrace \tau,\eta, \tau+\eta, \tau+2\eta \rbrace$ ($j_k<j_{k+1}$). Then the subsequence
$-({\rm Re}\beta^i_{j_k})^\vee, k=1,\dots,p$ is a concatenation of copies of two following sequences:
\begin{equation}\label{sequencecC_2}
\tau, \tau+2\eta, \tau+\eta,\tau+2\eta ~ {\rm and}~
\eta,\tau+\eta,\tau+2\eta.
\end{equation}
\end{cor}

Recall $r_i=t_{-\om_i}$. 
We denote by $t\mathcal{QB}(u,\la,\bar\beta^{i,\la})$ all alcove paths of type
$\bar\beta^{i,\lambda}=(\beta^i_1+{\rm len}({\rm Re}\beta_1^i)\langle {\beta^i_1}^\vee, \lambda \rangle\delta,
 \dots \beta^i_{r_i}+{\rm len}({\rm Re}\beta_r^i)\langle {\beta^i_r}^\vee, \lambda \rangle\delta)$
starting at $ut_{\la-\omega_i}$. In other words, instead of getting the sequence of roots $\beta_j$ from a reduced decomposition of $t_{-\om_i}$,
we start with the fixed sequence $\bar\beta^{i,\lambda}$.
\begin{thm}\label{combinatorialdecomposition}
Let $\lambda\in -X_+$. Then for $u\in W^a$ the following holds:
\[C_u^{t_{\lambda-\omega_i}}=\sum_{p \in t\mathcal{QB}(u,\lambda,\bar\beta^{i,\lambda})}q^{{\rm deg}({\rm qwt}(p))}
C^{t_{\lambda}}_{end(p)t_{-\lambda}}.\]
Further, if $u \in W$, then
\[C_u^{t_{\lambda-\omega_i}}=\sum_{p \in t\mathcal{QB}(u,\lambda,\bar\beta^{i,\lambda})}q^{{\rm deg}({\rm qwt}(p))}
C^{t_{\lambda}}_{dir(end(p))}x^{wt(end(p))-\lambda}.
\]
\end{thm}

\section{Generalized Weyl Modules}\label{GenWeylMod}
\subsection{Definitions and basic properties}
Let $\mathfrak{g}$ be a finite-dimensional simple Lie algebra and let $\tau$ be an automorphism of its Dynkin diagram  of order $j$.
We fix a primitive root of unity $\varepsilon$ of degree $j$ . Let
$\mathfrak{g}=\bigoplus_{k=0}^{j-1}\mathfrak{g}_k$, $\tau|_{\mathfrak{g}_k}=\varepsilon^k$ and let
 $\mathfrak{g}_{k+j} \simeq \mathfrak{g}_k$ for any $k \in \mathbb{Z}$.
We consider the twisted current algebra
\[\mathfrak{g}[t]^{(j)}=\bigoplus_{k=0}^\infty\mathfrak{g}_k\otimes t^k.\]
By definition we have the following inclusion:
\begin{equation}\label{inclusiontonontwisted}
\mathfrak{g}[t]^{(j)}\subset \mathfrak{g}[t].
\end{equation}

The space $\mathfrak{g}_0$ is a simple Lie algebra.
Let $\mathfrak{g}_0=\mathfrak{n}_-\oplus \mathfrak{h} \oplus \mathfrak{n}_+$ be the Cartan decomposition of $\fg_0$.
Let $\Delta$ be the root system of $\mathfrak{g}_0$. Denote by $\Delta^s$ the set of
short roots and by $\Delta^l$ the set of long roots in $\Delta$.
For a positive root $\al \in \Delta_+$ let $e_\al\in\fn_+$ and $f_{-\al}\in\fn_-$ be the Chevalley generators.
The weight lattice $X$ contains the positive part $X_+$; in particular, the fundamental weights
$\omega_1,\dots,\omega_n$ belong to $X_+$. For $\la\in X_+$ we denote by $V_\la$ the irreducible
highest weight $\fg_0$-module with highest weight $\la$.
The subalgebra $$\fn^{af}=\fn_+\oplus \bigoplus_{k=1}^\infty\mathfrak{g}_k\otimes t^k\subset\fg[t]$$
will play a crucial role in the constructions below.

Here is a list of diagram automorphisms: $\mathfrak{g}=A_{2n-1}$ and $j=2$, $\mathfrak{g}=D_{n+1}$ and $j=2$, 
$\mathfrak{g}=E_{6}$ and $j=2$, $\mathfrak{g}=D_{4}$ and $j=3$
and $\mathfrak{g}=A_{2n}, j=2$ (see Figure \ref{InclusionInvariants}).
In all the cases except for the last one the $\mathfrak{g}_0$ module $\mathfrak{g}_1$ has the following property:
${\rm wt}_{\mathfrak{h}^0}(\mathfrak{g}_1)=\Delta^s \cup \lbrace 0 \rbrace$, i. e. the set of nonzero weights of $\fg_1$ is equal to the
set of short roots of $\mathfrak{g}_0$. Let
$\delta$ be the basic imaginary root. We note that $\mathfrak{g}_k$ contains the nontrivial weight $k\delta$ subspace $\fh_{k\delta}$.
For each $\alpha\in\Delta$ the weight $\al+k\delta$, $k \geq 0$ subspace of $\fg^{(j)}[t]$ is at most one dimensional
(this subspace is trivial for $\al$ long and $k \notin j \mathbb{Z}$).
We fix a generator of this vector space (if nontrivial) and denote it by $e_{\alpha+k\delta}$.
So the set of real roots of the affine twisted Lie algebra (except for $A^{(2)}_{2n}$) is equal to
$(\Delta^l+ j\mathbb{Z}_{\geq 0}\delta)\cup (\Delta^s+ \mathbb{Z}_{\geq 0}\delta)$.
In the case $\mathfrak{g}=A_{2n}, j=2$ ($A_{2n}^{(2)}$) the $\mathfrak{g}_0$ module $\mathfrak{g}_1$ is
the highest weight module $V_{2\omega_1}$. The set of its nonzero weights is equal to
$\lbrace \alpha|\alpha \in \Delta \rbrace \cup \lbrace 2\alpha|\alpha \in \Delta^s \rbrace$, i.~e. all the roots
of $\fg_0$ and doubled short roots of $\fg_0$.

The classical Weyl modules $W(\la)$, $\la\in X_+$ for twisted Kac-Moody Lie algebras were defined in \cite{CFS}:
$W(\la)$ is a cyclic $\fg[t]^{(j)}$ module with generator $v$ subject to the following defining
relations:
\begin{gather}
\label{classicalweyl1}
h v=0, h\in \fh_{k\delta}, k \geq 1;\ h v=\la(h)v\text{ for all } h\in\fh;\\
\label{classicalweyl2}
e_{\alpha+k\delta} v=0, k \geq 0;\ f_{-\alpha}^{\langle \alpha^\vee, \lambda \rangle+1}v=0,
\text{ for all } \al\in\Delta_+.
\end{gather}

Let $\mathfrak{g}$ be of dual untwisted type.
Let $W$ be the Weyl group of $\mathfrak{g}_0$.
In what follows we use the following notation. For a root $\alpha$ let ${\rm len}(\alpha)$ be its normalized length, i. e. $1$ for short and
$2$ or $3$ for long roots.
For an element $\sigma \in W$ and $\alpha \in \Delta_+$ we
define the action  $\widehat\sigma$ of $\sigma$ on the set $\Delta_+\cup(\Delta_-^s+\delta)\cup(\Delta_-^l+j\delta)$ as follows:
\begin{gather*}
\widehat \sigma(\alpha)=   \begin{cases}
                             {\sigma(\alpha)}, & \text{ if } \sigma(\alpha) \in \Delta_+, \\
                             {\sigma(\alpha)}+ {\rm len}(\alpha)\delta, &\text{ if } \sigma(\alpha) \in \Delta_-,
                            \end{cases}\\
\widehat \sigma(-\alpha+{\rm len}(\alpha)\delta)=   \begin{cases}
                             {\sigma(-\alpha+{\rm len}(\alpha)\delta)}, & \text{ if } \sigma(-\alpha) \in \Delta_-, \\
                             {\sigma(-\alpha)}, & \text{ if } \sigma(-\alpha) \in \Delta_+.
                            \end{cases}.							
\end{gather*}
We also define the corresponding action on the root vectors $e_\beta$, $\beta\in \Delta_+\cup(\Delta_-^s+\delta)\cup(\Delta_-^l+j\delta)$:
$\widehat{\sigma} e_\beta=e_{\widehat{\sigma}(\beta)}$.

Using this action we define the generalized Weyl modules for dual untwisted types analogously to the untwisted definition from \cite{FM3}.

\begin{definition}\label{weylmodules}
For $\mu\in X$ write $\mu=\sigma(\lambda)$, $\sigma\in W$, $\lambda\in X_-$. Then the (generalized) Weyl module
$W_{\mu}$
is the cyclic $\fn^{af}$ module with a generator $v$ and the following relations:
\begin{gather}
\label{weylvanishing0}
h v=0, h\in \fh_{k\delta}, k \geq 1;\\
\label{weylvanishing1}
(\widehat{\sigma}(e_{-\alpha+{\rm len}(\alpha) \delta}) v=0;\\
\label{weylbound2}
(\widehat{\sigma}(e_{\alpha}))^{-\langle \alpha^\vee, \lambda \rangle+1} v=0
\end{gather}
for any $\alpha \in \Delta_+$.
\end{definition}

\subsection{Mixed type}
In this subsection we consider types $A^{(2)}_{2n}$. In this case $\mathfrak{g}_0$ is of type $B_n$. Then real roots
written in terms of orthonormal vectors $\epsilon_1,\dots,\epsilon_n$
are $\pm\epsilon_i+k\delta, \pm\epsilon_i\pm\epsilon_j+k\delta, \pm 2\epsilon_i+(2k+1)\delta$, $i,j \in 1, \dots, n$, $k \in \mathbb{Z}$.
Let $\Delta$ be the root system of type $B_n$.
We define the action of the Weyl group on the set $\Delta_+\cup(\Delta^l_-+\delta)\cup (2\Delta^s_-+\delta)$, i.e. on the following set:
\[\lbrace\epsilon_i\rbrace\cup\lbrace\epsilon_i\pm\epsilon_j|i>j\rbrace\cup\lbrace-\epsilon_i\mp\epsilon_j+\delta|i>j\rbrace
\cup\lbrace-2\epsilon_i+\delta\rbrace.\]

The action is defined by the formulas
\begin{gather*}
\widehat{\sigma}(\epsilon_i\pm\epsilon_j)=\begin{cases}
                             \sigma(\epsilon_i\pm\epsilon_j)+\delta, & \text{ if } \sigma(\epsilon_i\pm\epsilon_j) \in \Delta_- \\
                             \sigma(\epsilon_i\pm\epsilon_j), & \text{ if } \sigma(\epsilon_i\pm\epsilon_j) \in \Delta_+ \\
                            \end{cases},\\
\widehat{\sigma}(-\epsilon_i\mp\epsilon_j+\delta)=\begin{cases}
                             \sigma(-\epsilon_i\mp\epsilon_j)+\delta, & \text{ if } \sigma(-\epsilon_i\mp\epsilon_j) \in \Delta_- \\
                             \sigma(-\epsilon_i\mp\epsilon_j), & \text{ if } \sigma(-\epsilon_i\mp\epsilon_j) \in \Delta_+ \\
                            \end{cases}\\
\widehat{\sigma}(\epsilon_i)=\begin{cases}
                             2\sigma(\epsilon_i)+\delta, & \text{ if } \sigma(\epsilon_i) \in \Delta_- \\
                             \sigma(\epsilon_i), & \text{ if } \sigma(\epsilon_i) \in \Delta_+ \\
                            \end{cases},\\
\widehat{\sigma}(-2\epsilon_i+\delta)=\begin{cases}
                             -2\sigma(\epsilon_i)+\delta, & \text{ if } \sigma(-\epsilon_i) \in \Delta_- \\
                             \sigma(-\epsilon_i), & \text{ if } \sigma(-\epsilon_i) \in \Delta_+ \\
                            \end{cases},\\
\end{gather*}

We define $\widehat{\sigma} e_\beta=e_{\widehat{\sigma} (\beta)}$ for  $\beta\in \Delta_+\cup(\Delta^l_-+\delta)\cup (2\Delta^s_-+\delta)$.

\begin{definition}\label{weylmodulesmix}
For $\mu\in X$ write $\mu=\sigma(\lambda)$, $\sigma\in W$, $\lambda\in X_-$. Then the (generalized) Weyl module
$W_{\mu}$
is the cyclic $\fn^{af}$ module with a generator $v$ and the following relations:
\begin{gather}
\label{weylvanishingmix0}
h v=0, h\in \fh_{k\delta}, k \geq 1;\\
\label{weylvanishingmix1}
(\widehat{\sigma}(e_{-\alpha+ \delta})) v=0, \al\in\Delta_+^l\cup 2\Delta_+^s;\\
\label{weylboundmix2}
(\widehat{\sigma}(e_{\alpha}))^{-\langle \alpha^\vee, \lambda \rangle+1} v=0, \ \alpha\in\Delta_l^+,\\
\label{weylboundmix3}
(\widehat{\sigma}(e_{\alpha}))^{-{\langle \alpha^\vee, \lambda \rangle}+1} v=0, \ \alpha\in\Delta_s^+ \text{ and } \sigma(\alpha) \in \Delta_+.\\
\label{weylboundmix4}
(\widehat{\sigma}(e_{\alpha}))^{\left[-\frac{\langle \alpha^\vee, \lambda \rangle}{2}\right]+1} v=0, \alpha\in\Delta_s^+ \text{ and } \sigma(\alpha) \in \Delta_-.
\end{gather}
\end{definition}

\begin{rem}
The restriction of $W(\lambda)$ to $\fn_+^{af}$ is isomorphic to $W_{-\lambda}$,
$\la\in X_+$ (since $w_0\la=-\la$ for non simply-laced algebras).
\end{rem}

\subsection{Dual mixed type}
Let us consider the Dynkin diagram of affine type $A_{2n}^{(2)}$:
$$
\xymatrix{
\overset{\alpha_0}{\circ}  \ar@{=>}[r]  & \overset{\alpha_1}{\circ}  \ar@{-}[r] & \overset{\alpha_2}{\circ}
\ar@{-}[r]& \overset{\alpha_3}{\circ}  \ar@{.}[r] & \overset{\alpha_{n-3}}{\circ}  \ar@{-}[r]&
\overset{\alpha_{n-2}}{\circ}  \ar@{-}[r]& \overset{\alpha_{n-1}}{\circ}  \ar@{=>}[r] &\overset{\alpha_n}{\circ}
}
$$
We have two natural possibilities to choose a special vertex. If we call the leftmost vertex special, then
we obtain the usual $A_{2n}^{(2)}$ picture with $\mathfrak{g}_0=B_n$.
Denote by $\Delta$ the root system of this zero level algebra.
If we call the rightmost vertex special, then the
underlying simple Lie algebra is of type $C_n$ and will be denoted by $\mathfrak{g}_0^\dagger$. 
The corresponding twisted current algebra was considered in \cite{CIK} (the algebra is referred to as
the hyperspecial parabolic subalgebra).  In the second case we have 
$\mathfrak{g}_1^\dagger\simeq V_{\omega_1}$, $\mathfrak{g}_2^\dagger\simeq V_{\omega_2}$,
$\mathfrak{g}_3^\dagger\simeq V_{\omega_1}$, $\mathfrak{g}_{i+4}^\dagger\simeq \mathfrak{g}_i^\dagger$ as $\mathfrak{g}_0^\dagger$-modules.
Let $\Delta^{\dag}$ be the root system of $\fg_0^\dagger$. Then the weights of $V_{\omega_1}$ can
be identified with halves of the long roots of $\mathfrak{g}_0^\dagger$ and  the weights of $V_{\omega_2}$ can
be identified with short roots of $\mathfrak{g}_0^\dagger$. Therefore the roots of affine algebra can be written in
the following way:
\[\lbrace \Delta^{\dag} + 4 \mathbb{Z} \delta\rbrace \cup \lbrace \Delta^{\dag s} + (2+4 \mathbb{Z}) \delta\rbrace \cup
\lbrace \frac{1}{2}\Delta^{\dag l} + (1+2 \mathbb{Z}) \delta\rbrace.\]

For a dominant $\lambda$ we denote by $W^{\dag}({\lambda})$  the hyperspecial Weyl modules from \cite{CIK}. So $W^{\dag}({\lambda})$
is the cyclic module with the following relations:
\begin{gather*}
\fh_{k\delta} v=0, k \geq 1;\ e_{\alpha+4k \delta} v=0;~ e_{1/2 \alpha+(1+2k) \delta}v=0, \alpha \in \Delta_+^l, k\ge 0;\\
e_{\alpha+2k \delta} v=0,\ \alpha \in \Delta_+^s, k\ge 0;\ f_{-\alpha}^{\langle \alpha, \lambda \rangle+1}v=0, \alpha \in \Delta_-.
\end{gather*}
One can define the generalized Weyl modules in this case as well. However it does not make sense because
the generalized Weyl modules of dual mixed type are isomorphic to the generalized Weyl modules of mixed type.
To prove this claim, we consider the element $w^\dag$ of the Weyl group of type $B_n$ sending
$\alpha_i$ to $\al_{n-i}$, $1\le i\le n-1$ and $w^\dag\al_n=-\al_1-\dots -\al_n$ ($w^\dag$ sends $\epsilon_i$ to $-\epsilon_{n+1-i}$).

Let $[x]$ denotes the floor function of $x$. 
\begin{prop}\label{dagomega}
Let $\lambda=-\sum_{i=1}^n m_i \omega_i$, $\lambda^{\dag}=-\sum_{i=1}^{n-1} m_i \omega_{n-i}- \left[\frac{m_n}{2} \right]\omega_0$.
Then $W_{w^\dag(\lambda)}\simeq W^{\dag}(-\lambda^{\dag})$.
\end{prop}
\begin{proof}
We have that $\alpha_0=-2(\alpha_1+\dots +\alpha_n)+\delta$ and $\alpha_0=\widehat{w^\dag}(\alpha_n)$.
Moreover $\Delta_+^{\dag}=\widehat{w^\dag}(\Delta_+)$ and
$\widehat{w^\dag}(\alpha_i)=\alpha_{n-i}$, $i=1, \dots, n-1$. Using \eqref{weylboundmix4} we have:
\[e_{\alpha_0}^{\left[ \frac{m_n}{2} \right]+1}v=0.\]
This completes the proof of the Proposition.
\end{proof}

\subsection{Properties of $W_{\sigma(\lambda)}$}
It is easy to see that the generalized Weyl modules are well defined, i.e. $W_\mu$ does
not depend on the choice of $\sigma$ and $\la$ such that  $\sigma(\lambda)=\mu$ (see Lemma 2.2 in \cite{FM3}).

We note that the algebra $\fn^{af}$ does not contain the finite Cartan subalgebra $\fh$. However, sometimes
it is convenient to add extra operators from $\fh$ acting on $W_\mu$.
\begin{definition}
For $\nu\in X$ we define $W_{\mu}^\nu$ to be the $\fn^{af}\oplus\fh$-module where the action of
$\fh$ is defined by $hv=\nu(h)v$ for all $h\in\fh$ and the cyclic vector $v$.
If $\nu=\mu$, we omit the upper index and write $W_{\mu}$ for $W_{\mu}^\mu$.
\end{definition}

The modules $W_\mu^\nu$ are naturally graded by the Cartan subalgebra $\fh$ and carry additional
degree grading defined by two conditions: ${\rm deg} (v)=0$ and the operators $e_{\gamma+k\delta}$ increase
the degree by $k$. We define the character by the formula:
\[
{\rm ch} W_\mu^\nu=\sum \dim W_\mu^\nu[\gamma,k] x^\gamma q^k ,
\]
where $W_\mu^\nu[\gamma,k]$ consists of degree $k$ vectors of $\fg_0$-weight $\gamma$.
In particular, we write ${\rm ch} W_\mu$ for the character of $W_\mu^\mu$.

\begin{figure}
\xymatrixrowsep{0.09in}
\xymatrixcolsep{0.2in}
\xymatrix{
\overset{e_1}{\circ} \ar@{-}[r]  & \overset{e_2}{\circ} \ar@{-}[r] & \overset{e_3}{\circ} \ar@{.}[r] &
\overset{e_{n-2}}{\circ} \ar@{-}[r] &\overset{e_{n-1}}{\circ} \ar@{-}[dr]& \\
&&&&&\overset{e_{n}}{\circ} \ar@{-}[dl]\\
\underset{e_{2n-1}}{\circ} \ar@{-}[r] & \underset{e_{2n-2}}{\circ} \ar@{-}[r] & \underset{e_{2n-3}}{\circ} \ar@{.}[r]
& \underset{e_{n+2}}{\circ} \ar@{-}[r] &\underset{e_{n+1}}{\circ}& \\
\underset{\frac{e_1+ e_{2n-1}}{2}}{\circ} \ar@{-}[r]  & \underset{\frac{e_2+ e_{2n-2}}{2}}{\circ} \ar@{-}[r]
 & \underset{\frac{e_3+ e_{2n-3}}{2}}{\circ} \ar@{.}[r] &\underset{\frac{e_{n-2}+ e_{n+2}}{2}}{\circ} \ar@{-}[r] &
\underset{\frac{e_{n-1}+ e_{n+1}}{2}}{\circ} \ar@{<=}[r] &\underset{{e_n}}{\circ} &C_n \subset A_{2n-1}\\
&&&&& \overset{e_n}{\circ} \\
\overset{e_1}{\circ} \ar@{-}[r]  & \overset{e_2}{\circ} \ar@{-}[r] & \overset{e_3}{\circ} \ar@{.}[r] &
\overset{e_{n-2}}{\circ} \ar@{-}[r] &\overset{e_{n-1}}{\circ} \ar@{-}[dr] \ar@{-}[ur]& \\
&&&&& \overset{e_{n+1}}{\circ} \\
\overset{e_1}{\circ} \ar@{-}[r]  & \overset{e_2}{\circ} \ar@{-}[r]
 & \overset{e_3}{\circ} \ar@{.}[r] &\overset{e_{n-2}}{\circ} \ar@{-}[r] &
\overset{e_{n-1}}{\circ} \ar@{=>}[r] &\overset{\frac{e_{n}+ e_{n+1}}{2}}{\circ} &B_n \subset D_{n+1}\\
\overset{e_1}{\circ} \ar@{-}[r]  & \overset{e_3}{\circ} \ar@{-}[rd]&&&\overset{e_1}{\circ} \ar@{-}[rd]\\
&&\overset{e_4}{\circ} \ar@{-}[r]  & \overset{e_2}{\circ}&\overset{e_3}{\circ} \ar@{-}[r]&\overset{e_2}{\circ}  \\
\overset{e_6}{\circ} \ar@{-}[r]  & \overset{e_5}{\circ} \ar@{-}[ru]&& F_4 \subset E_{6}&\overset{e_4}{\circ} \ar@{-}[ru] \\
\overset{\frac{e_1+ e_{6}}{2}}{\circ} \ar@{-}[r]  & \overset{\frac{e_3+ e_{5}}{2}}{\circ} \ar@{<=}[r]
 & \overset{e_4}{\circ} \ar@{-}[r] &\overset{e_2}{\circ} & \overset{\frac{e_1+ e_3+e_4}{3}}{\circ} \ar@3{<-}[r]&\overset{e_2}{\circ}
 & G_2 \subset D_{4}\\
\overset{e_1}{\circ} \ar@{-}[r]  & \overset{e_2}{\circ} \ar@{-}[r] & \overset{e_3}{\circ} \ar@{.}[r] &
\overset{e_{n-1}}{\circ} \ar@{-}[r] &\overset{e_n}{\circ} \ar@{-}[d]& \\
\underset{e_{2n}}{\circ} \ar@{-}[r] & \underset{e_{2n-1}}{\circ} \ar@{-}[r] & \underset{e_{2n-2}}{\circ} \ar@{.}[r]
& \underset{e_{n+2}}{\circ} \ar@{-}[r] &\underset{e_{n+1}}{\circ}& \\
\overset{\frac{e_1+ e_{2n}}{2}}{\circ} \ar@{-}[r]  & \overset{\frac{e_2+ e_{2n-1}}{2}}{\circ} \ar@{-}[r]
 & \overset{\frac{e_3+ e_{2n-2}}{2}}{\circ} \ar@{.}[r] &
\overset{\frac{e_{n-2}+ e_{n+1}}{2}}{\circ} \ar@{=>}[r] &\overset{\frac{e_{n-1}+ e_n}{2}}{\circ} &B_n \subset A_{2n}
}
\caption{Subalgebras $X_{n}^0 \subset X_{n}$}\label{InclusionInvariants}
\end{figure}
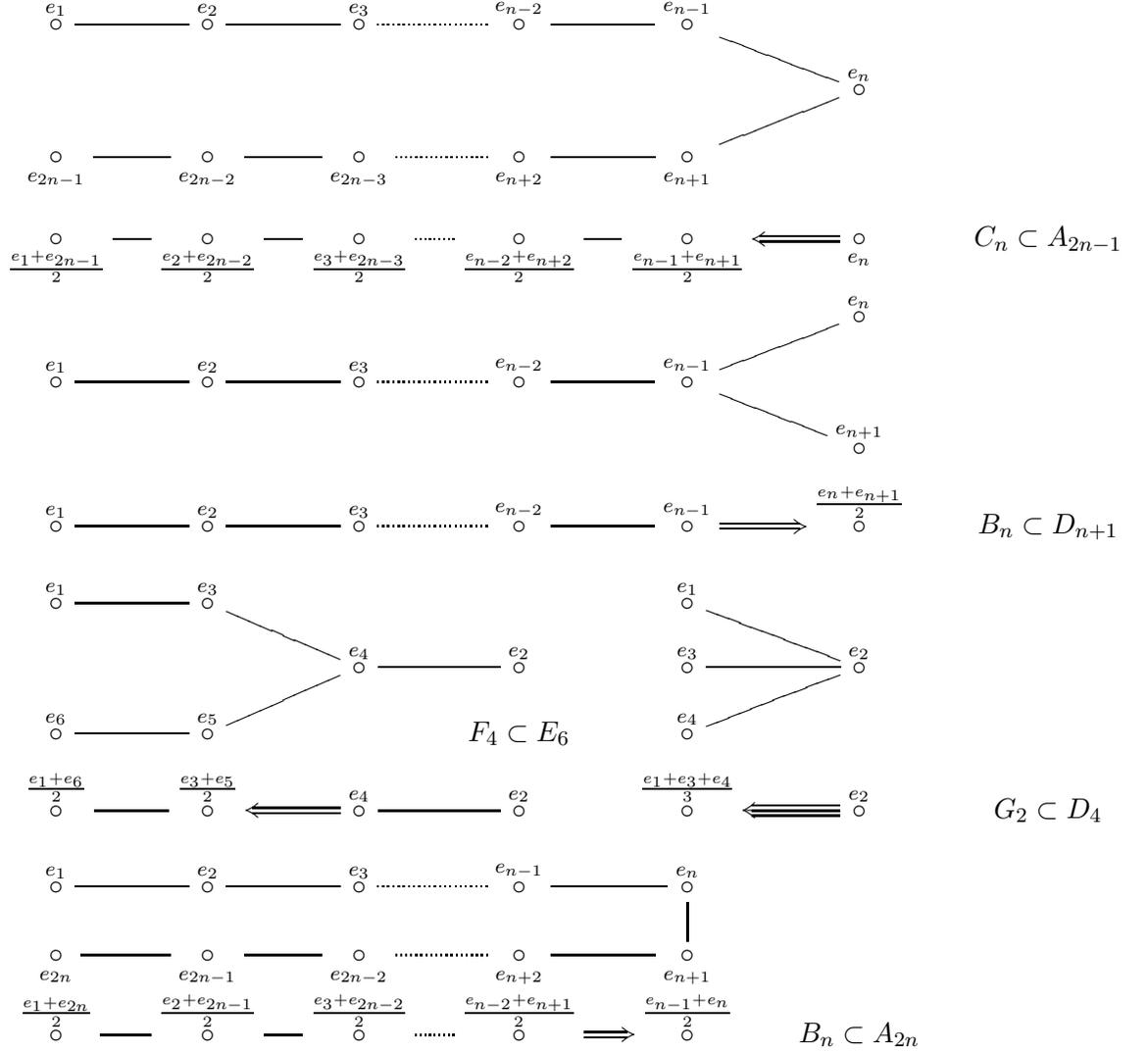

\subsection{Fusion-like construction}
Our first goal is to define the action of $\fn^{af}$ on the Weyl modules of the corresponding untwisted algebra
$\fg[t]$. Let $\lambda$ be an antidominant weight of $\fg$ and
let $W(\bar{w_0}\la)$ be the (classical) Weyl module over $\fg[t]$ with a lowest weight vector $v$ 
(where $\bar{w_0}$ is the longest element of the Weyl group of the algebra $\fg$).
Recall the inclusion \eqref{inclusiontonontwisted}.
Let $z$ be some nonzero complex number.  For $u \in W(\bar w_0\lambda)$, $e_{\alpha} \otimes t^k \in \fg[t]^{(j)}$
we consider the shifted action:
\[e_{\alpha} \otimes t^k. u= e_{\alpha} \otimes (t+z)^ku.\]

Let $\sigma$ be an element of the Weyl group of the algebra $\mathfrak{g}_0$. 
Since $\fg_0\subset\fg$, the Weyl group of $\fg_0$ acts projectively on the irreducible $\fg$ module $V_\la\subset W(\bar{w_0}\la)$.
We consider a vector $\sigma v\in W(\bar{w_0}\la)$ defined up to a scalar.
Then it is easy to see using the Vandermonde determinant that the cyclic $\fn^{af}$-submodule
${\rm U}(\fn^{af})\sigma(v)$ coincides with the whole module $W(\bar{w_0}\la)$. 
We consider the following filtration on this module:
$F_0=\mathbb{C}\sigma(v)$ and
\[F_r={\rm span}\{(x_1\T t^{s_1}\dots x_l\T t^{s_l}).\sigma(v)\},
s_1+\dots+s_l\le r, x_a\T t^{s_a}\in \fn^{af}\}.\]
Let $W_{\sigma,\lambda}^{\star}$ be the corresponded graded module.

To stress the difference between the weights of $\fg$ and $\fg_0$ we denote by $\om_i$ the fundamental weight of $\fg$ and by 
$\omega_{0i}$, $i=1,\dots,n$, $n={\rm rk}(\fg_0)$ the fundamental weights of $\fg_0$. We use this notation in this subsection only.
\begin{lem}\label{lowerbound}
Let $\lambda=-\sum_{i=1}^nm_i \omega_i$ be an antidominant weight of the algebra $\mathfrak{g}$.
Let $\lambda_0$ be the antidominant weight of $\mathfrak{g}_0$ defined by:
$\lambda_0=-\sum_{i=1}^nm_i \omega_{0i}$ for dual untwisted types and
$\lambda_0=-\sum_{i=1}^{n-1} m_i \omega_{0i}- 2m_n \omega_{0n}$ for types $A^{(2)}_{2n}$
(see  Figure \ref{InclusionInvariants} for the numbering). We consider the $\fg[t]$-module $W(\bar w_0\lambda)$. Then for any element $\sigma$
of the Weyl group $W$ of the Lie algebra $\mathfrak{g}_0$ we have:
\begin{equation}\label{dimensioninequation}
\dim W_{\sigma(\lambda_0)}\geq \dim W(\bar {w_0}\lambda).
\end{equation}
\end{lem}
\begin{proof}
We show that there exists an epimorphism $W_{\sigma(\lambda_0)} \rightarrow W_{\sigma,\lambda}^{\star}$.
So we need to prove that all the defining relations of the module $W_{\sigma(\lambda_0)}$ hold in the module $W_{\sigma,\lambda}^{\star}$.
For a diagram automorphism $\tau$ and a simple root vector $e_{i}$ except for the root $e_n$ of $A_{2n}^{(2)}$ we have:
\[[e_{i},\tau e_{i}]=0,\]
(see Figure \ref{InclusionInvariants}). In type $A_{2n}^{(2)}$ the vectors
$e_n, \tau e_n,[e_{i},\tau e_i]$ form a Heisenberg algebra.

Let $v$ be the lowest weight vector of $W(\bar{w_0}\lambda)$.
Note that any orbit of the group generated by $\tau$ contains exactly one element $e_i$ such that $i \leq n$.
We have $\widehat{\sigma} (e_{\alpha + \delta}).\sigma(v)=0$, $\alpha \in \Delta_-$.
For all cases except for $i=n$ in type $A_{2n}^{(2)}$ we have:
\[\left(\frac{\sum_{k=0}^{j-1} \tau^k e_{i}}{j}\right)^{m_i+1}.v=0.\]
Analogously for $A_{2n}^{(2)}$ we have: $\left(e_n+\tau e_n\right)^{2m_n+1}. v=0.$
Hence for any positive root $\alpha \in \Delta_+$
\[ e_{\alpha}^{-\langle \lambda_0, \alpha^\vee \rangle+1}.v=0.\]
Thus using the action of the Weyl group we obtain
\[ e_{\sigma(\alpha)}^{-\langle \lambda_0, \alpha^\vee \rangle+1}.\sigma(v)=0.\]

Thus rank-one computation shows that
\[\widehat{\sigma}( e_{\alpha})^{-\langle \lambda_0, \alpha^\vee \rangle+1}.\sigma(v)=0\]
and
\[\widehat{\sigma}( e_{\alpha})^{\left[-\frac{\langle \lambda_0, \alpha^\vee \rangle}{2}\right]+1}.\sigma(v)=0\]
if $\fg$ is of type $A^{(2)}_{2n}$, $\alpha$ is a short root and $\sigma(\alpha)\in \Delta_-$.

So we have proved that the defining relations of $W_{\sigma(\lambda_0)}$ hold in $W(\bar{w_0}\lambda)$ with the shifted action. Hence
they also hold in $W_{\sigma,\lambda}^{\star}$. This completes the proof of the Lemma.
\end{proof}

\subsection{tQBG and Weyl modules}
In the following lemma we give a criterion of the existence of edges in the twisted quantum Bruhat graph.
In part $ii)$ by a long root we mean a root such that there exists another root of a smaller length.

\begin{lem}\label{edgesinQBG}
Consider the dual untwisted case.
For $\sigma \in W$, $\gamma \in \Delta_+$ the two following statements are equivalent:

$i)$ there is an edge in the twisted quantum Bruhat graph $\sigma \stackrel{\gamma}{\longrightarrow} \sigma s_{\gamma}$;

$ii)$ there are no elements $\alpha, \beta \in \Delta_+$ such that
$\alpha, \beta \neq \gamma$, $\alpha + \beta=2\frac{\langle \alpha, \gamma \rangle}{\langle \gamma, \gamma \rangle}\gamma$,
$\widehat{\sigma}(\alpha) + \widehat{\sigma}(\beta)=
2\frac{\langle \alpha, \gamma \rangle}{\langle \gamma, \gamma \rangle}\widehat{\sigma}(\gamma)$;
if $\sigma\gamma\in\Delta_-$, then additionally
$\gamma$ is not a long nonsimple root contained in a rank two subalgebra, generated by roots from $\Delta_+$.
\end{lem}

\begin{proof}
The proof is analogous to the proof for untwisted algebras by interchanging roots and coroots (see \cite{FM3}, Lemma 2.12).
\end{proof}

\begin{definition}\label{betasWeyl}
Let $\bar\beta=(\beta_1, \dots, \beta_r)$ be a sequence of affine roots.\label{quasiweylmodules}
Consider a weight $\sigma(\lambda)$, where $\sigma \in W$, $\lambda\in X_-$.
Then the generalized Weyl module with characteristics $W_{\sigma(\lambda)}(\bar\beta,m)$, $m=0,\dots,r$ is the cyclic $\fn^{af}$
module with a generator $v$ and the following relations:
$h v=0, h\in \fh_{k\delta}, k \geq 1$ and for all  $\al\in\Delta_+$: $\widehat\sigma(e_{-\alpha+ \delta}) v=0$ and
\begin{gather*}
\widehat{\sigma}(e_{\alpha})^{\left[\frac{l_{\al,m}}{2}\right]+1}v=0, \text{ if } \fg \text{ of type } A^{(2)}_{2n},
\sigma(\alpha) \in \Delta_-,\ \al \text{ is short};\\
\widehat{\sigma}(e_{\alpha})^{l_{\al,m}+1}v=0, \text{ otherwise},
\end{gather*}
where $l_{\al,m}= -{\langle \alpha^\vee, \lambda \rangle-|\lbrace \beta_i|{\rm Re} \beta_i=-\alpha, i \leq m \rbrace|}.$
\end{definition}

\begin{rem}
If $m=0$, then $W_{\sigma(\lambda)}(\bar\beta,0)\simeq W_{\sigma(\lambda)}$. Now assume that $m=l(t_{-\om_i})$ and
the sequence of roots $\bar\beta$ comes from a reduced decomposition of $t_{-\omega_i}$. Then according to Proposition
\ref{descriptionbeta}, part $a)$, we have an isomorphism
\[
W_{\sigma(\lambda)}(\bar\beta,m)\simeq W_{\sigma(\lambda+\omega_i)}.
\]
\end{rem}


Let $\tau_1, \tau_2 \in \Delta_+$ be two roots from the root system of $\fg_0$. Let $L_2$ be the semisimple Lie algebra with the root system
$\Delta'\subset \Delta$ spanned by roots $\tau_1, \tau_2$. Let ${\fn^{af}}'$ be the subalgebra of $\fn^{af}$ generated
by the elements of the form $e_{\gamma+ k \delta}\in \fn^{af}$, where $\gamma\in \Delta'\cup 2\Delta'$, $k \geq 0$.
\begin{rem}
The set $2\Delta'$ shows up only in the $A^{(2)}_{2n}$ case.
\end{rem}

\begin{lem}\label{subsystemofrank2}
For the $\fn^{af}$-module
$W_{\sigma(\lambda)}(\bar\beta,m)$ we define ${\fn^{af}}'$-submodule
$M_2={\rm U}({\fn^{af}}')v\subset W_{\sigma(\lambda)}(\bar\beta,m)$, where
$v$ is the cyclic vector and $m$ satisfies $\sigma({\rm Re}\beta_{m+1})\in \mathbb{Z}\langle \tau_1, \tau_2\rangle$.
Then $M_2$ is a quotient of some ${\fn^{af}}'$ module of the form $W_{\widetilde \sigma(\widetilde \lambda)}(\widetilde\beta,\widetilde m)$,
where $\widetilde \sigma$, $\widetilde \lambda$, $\widetilde\beta$, $\widetilde m$ are parameters for ${\fn^{af}}'$.
In addition,  $\sigma {\rm Re} \beta_{m+1}=
\widetilde\sigma{\rm Re} \widetilde{\beta}_{\widetilde m+1}$.
\end{lem}
\begin{proof}
Without loss of generality we assume that $\tau_1, \tau_2$ is a basis of $\Delta'=\mathbb{Z}\langle \tau_1, \tau_2\rangle \cap \Delta$.
We have several possible cases  for $\Delta'$ and ${\fn^{af}}'$. The first case: $\Delta'$ is the root system of type $A_1\oplus A_1$
(this happens when $\tau_1$ is orthogonal to $\tau_2$ and $[e_{\tau_1},e_{\tau_2}]=0)$. In this case
the claim is obvious. The second case is $L_2 \simeq G_2$ and ${\fn^{af}}'$ isomorphic to the positive affine nilpotent algebra
of type $D_4^{(3)}$. So we have $\fn^{af}={\fn^{af}}'$ and hence there is nothing to prove. The third case is $\Delta'$ being
the root system of type $A_2$; then ${\fn^{af}}'$ is the positive affine nilpotent algebra
of type $A_2^{(1)}$. This case was worked out in \cite{FM3}, Lemma 2.17. The two remaining cases are $A_3^{(2)}$ and $A_4^{(2)}$
(so $\Delta'$ is either of type $C_2$ or of type $B_2$). The $A_3^{(2)}$ is equivalent to the untwisted $C_2$ situation from
\cite{FM3}. The only difference is that we apply Corollary \ref{order}, $iii)$, to the roots instead of coroots.
Finally, the last $A_4^{(2)}$ case is equivalent to the $A_3^{(2)}$ thanks to the recipe from \cite{OS} (see also subsection \ref{alcovepaths}).
\end{proof}

\subsection{The decomposition procedure}
In order to describe the generalized Weyl modules we follow the same strategy as in \cite{FM3}. Namely,
we consider the sequence of surjections
\[
W_{\sigma(\la)} = W_{\sigma(\la)}(\bar\beta^i,0)\to W_{\sigma(\la)}(\bar\beta^i,1)\to\dots\to
W_{\sigma(\la)}(\bar\beta^i,r)= W_{\sigma(\la+\om_i)},
\]
$\bar\beta^i=\bar\beta(t_{-\omega_i})$, $\langle \lambda, \alpha_i^\vee \rangle \neq 0$
and describe the kernels. The description is given in terms of the twisted quantum Bruhat graph for the twisted
affine Kac-Moody algebra.

\begin{rem}\label{A2rec}
The type $A^{(2)}_{2n}$ is very special. First, the tQBG in this case is of type $D^{(2)}_{n+1}$.
Second, in what follows we will be interested in the paths along the edges of tQBG.
More concretely, given a vertex $\sigma$ of tQBG and an affine root $\beta$ we will check if there is an
edge $\sigma\to \sigma s_{{\rm Re}\beta}$. The existence of such an edge allows us to move
along the graph. However, ONLY in types $A^{(2)}_{2n}$, there is an additional restriction (see \cite{OS}).
Namely, even if an edge does exist, we are not allowed to use it if the edge is quantum and
${\rm deg}\beta^i_{m+1}+\langle {\beta^i_{m+1}}^\vee,\lambda + \omega_i \rangle$ is odd (see Subsection \ref{alcovepaths}, Lemma \ref{firstbetas}).
\end{rem}

\begin{thm}\label{stepofdecomposition}
Let $\bar\beta^i=(\beta_1^i,\dots,\beta_r^i)$ be the sequence of $\beta$'s for some reduced decomposition of the element $t_{-\omega_i}$.
Then we have:

$(i)$ If there is no edge $\sigma \stackrel{{\rm Re} \beta_{m+1}^i}{\longrightarrow} \sigma s_{{\rm Re} \beta_{m+1}}$
or in type $A_{2n}^{(2)}$ the edge $\sigma \stackrel{{\rm Re} \beta_{m+1}^i}{\longrightarrow} \sigma s_{{\rm Re} \beta_{m+1}^i}$
is quantum and ${\rm deg}\beta^{i,\lambda+\omega_i}_{m+1}={\rm deg} \beta^i_{m+1}+\langle {\beta^i_{m+1}}^\vee,\lambda + \omega_i \rangle$
 is odd, then
\[W_{\sigma(\lambda)}(\bar\beta^i,m) \simeq W_{\sigma(\lambda)}(\bar\beta^i,m+1).\]

$(ii)$ Assume there is an edge $\sigma \stackrel{{\rm Re} \beta_{m+1}^i}{\longrightarrow} \sigma s_{{\rm Re}\beta_{m+1}}$.
Then we have the exact sequence
\[
W_{{\sigma}s_{{\rm Re} \beta_{m+1}}(\lambda)}(\bar\beta^i,m+1)\to  W_{\sigma(\lambda)}(\bar\beta^i,m)\to
W_{\sigma(\lambda)}(\bar\beta^i,m+1)\to 0.
\]
\end{thm}
\begin{proof}
As in \cite{FM3}, we reduce the proof to the rank two case using Lemmas \ref{edgesinQBG} and \ref{subsystemofrank2}.
\end{proof}

\begin{cor}\label{decompositioninequality}
\[{\rm ch} W_{\sigma(\lambda-\om_i)}\le \sum_{p\in \mathcal{QB}(\sigma,\lambda,\bar\beta^{i,\lambda})}
q^{{\rm deg}({\rm qwt}(p))} {\rm ch} W_{dir({\rm end}(p))(\la)}^{wt ({\rm end} (p))},\]
\[{\rm ch} W_{\sigma(\lambda)}\le C_\sigma^{t_{\lambda}},\]
where inequalities mean the coefficient-wise inequalities.
\end{cor}
\begin{proof}
The proof is the same as the proof of Corollary 2.19 from \cite{FM3}.
\end{proof}

\begin{thm}\label{equalityininequatity}
The inequalities of Corollary \ref{decompositioninequality} are in fact the equalities.
\end{thm}
\begin{proof}
According to \cite{I} in dual untwisted types $\dim W_{-\omega_i}=E_{-\omega_i}(1,1,0)$ for all fundamental weights $\omega_i$.
In type $A^{(2)}_{2n}$ \cite{I} implies $\dim W_{-w^\dag\omega_i}=E_{-\omega_{n-i}}(1,1,0)$, $i\ne n$ and
$\dim W_{-w^\dag 2\omega_n}=E_{-\omega_0}(1,1,0)$. Now the rest of the proof follows the same lines as in \cite{FM3}.
The only special case is $A^{(2)}_{2n}$ and $i=n$. Then inequality \eqref{dimensioninequation} is valid only for
$\lambda$ such that $\langle \lambda, \alpha_n^\vee \rangle$ is even. So we are forced to use double decomposition procedure
(i.e. we consider the filtration such that all the subquotients are isomorphic to $W_{\kappa(\lambda +2\omega_n)}$).
However, since we know that inequalities of Corollary \ref{decompositioninequality} become equalities after 
the double application of the decomposition procedure, we obtain the equalities at each application. 
\end{proof}

\begin{thm}
Let $\la$ be an antidominant weight, $\sigma\in W$. Then ${\rm ch} W_{\sigma(\la)}=C_{\sigma}^{t_\la}$.
In addition, the sequence
\[
0\to W_{{\sigma}s_{{\rm Re} \beta_{m+1}}(\lambda)}(\bar\beta^i,m+1)\to  W_{\sigma(\lambda)}(\bar\beta^i,m)\to
W_{\sigma(\lambda)}(\bar\beta^i,m+1)\to 0
\]
is exact.
\end{thm}
\begin{proof}
The proof is similar to the proof of Theorem 2.21, \cite{FM3} modulo the following equality
\begin{equation}\label{ldeg}
{\rm len} ({\rm Re}\beta_{m+1}^i) l_{-{\rm Re}\beta_{m+1}^i,m}=\deg\beta_{m+1}^i
\end{equation}
(in the untwisted settings the factor ${\rm len} ({\rm Re}\beta_{m+1}^i)$ is missing in the left hand side).
\end{proof}

As a consequence, we obtain an alternative proof of the following claim (see \cite{CI}).
\begin{cor}\label{Iongeneralization}
Let $\lambda$ be a dominant weight. Then for any twisted current algebra $E_{-\lambda}(x,q,0)=\ch W(\lambda)$.
\end{cor}

We also obtain the representation-theoretic interpretation of the specialization of nonsymmetric Macdonald polynomials at $t=\infty$.

\begin{cor}\label{representationinterpretationinfinity}
Let $\lambda$ be an antidominant weight. Then $E_{\lambda}(x,q^{-1},\infty)=\ch W_{-\lambda}.$
\end{cor}
\begin{proof}
We use  Theorem  \ref{specializationOS0},(ii)  and the following result (see e.g. \cite{LNSSS2}):
the element  $w_0 \in W$ inverses arrows in the twisted quantum Bruhat graph.
\end{proof}

\begin{cor}
Consider the twisted current algebra of type $A_{2n}^{(2)}$. Then  for any $\sigma\in W$ one has
\[
\dim W_{\sigma\left(\sum_{i=1}^{n-1} m_i\om_i+2m_n\om_n\right)}=\prod_{i=1}^n \binom{2n+1}{i}^{m_i}.
\]
In particular, we obtain the dimension \eqref{A2dim} for the classical Weyl modules.
\end{cor}
\begin{proof}
The decomposition procedure tells us that
$$\dim W_{\sigma\left(\sum_{i=1}^{n-1} m_i\om_i+2m_n\om_n\right)}=(\dim W_{-2w^{\dag}\omega_n})^{m_n}\prod_{i=1}^{n-1}
(\dim W_{-w^{\dag}\omega_i})^{m_i}.$$
However using Proposition \ref{dagomega} and \cite{CIK}, Theorem 3 we know that $\dim W_{-w^{\dag}\omega_i}=\binom{2n+1}{i}$, $i<n$ and $\dim W_{-2w^{\dag}\omega_n}=\binom{2n+1}{n}$.
\end{proof}
\begin{rem}
Assume that $\langle\la,\al_n^\vee\rangle \in 2 \bZ+1$. Then the dimension of $W_{\sigma{\lambda}}$ depends on $\sigma$
(this already happens in type $A_2^{(2)}$, see subsection \eqref{A22}).
For $\sigma={\rm id}$ this dimension is known (see \cite{FK}).
\end{rem}

\section{Low rank cases}\label{LRC}
In this section we work out the case of rank one and rank two level zero algebras. The corresponding twisted affine algebras are of types
$A_2^{(2)}$, $D_3^{(2)}$, $A_4^{(2)}$ and $D_4^{(3)}$. We first prepare several facts to be used in the proofs.
\begin{rem}\label{ndelta}
In what follows we use the relation in $W_{\sigma(\lambda)}(\bar\beta^i,m)$: for any $\al\in\Delta_+$ and $n\ge 0$
\[
e_{\widehat{\sigma}(\al)+n\delta}^{l_{\al,m}+1}v=0.
\]
\end{rem}

\begin{rem}\label{BGG}
Let $\fg$ be a simple Lie algebra and $\fn\subset\fg$ its nilpotent subalgebra. Let $V$ be a cyclic $\fn$ module
with cyclic vector $v$ such that $e_{\al_i}^{m_i+1}v=0$ for any simple root $\al_i$, $m_i\ge 0$. Then
for any positive root $\al$ one has $e_\al^{\langle \sum m_i\om_i,\al^\vee\rangle+1}v=0$.
\end{rem}

\begin{lem}\label{simpleroots}
$(i)$Let $\mathfrak{g}$ be of arbitrary type. Let $\tau$ be a positive root and for $m\ge 0$ let $\eta=-{\rm Re} \beta^i_{m+1}$.
Assume that  $\eta$ and $s_{\eta}\tau$ are simple.
Then $m=0$ and
\begin{equation}\label{eqinBGG}
\widehat{\sigma s_{\eta}}(e_{\tau})^{l_{\tau,m+1}+1}
\widehat{\sigma}(e_{\eta})^{l_{\eta,m}} v =0.
\end{equation}
$(ii)$ If $m=0$ and there exists a subalgebra $N \subset \fn^{af}$ isomorphic to the finite dimensional
nilpotent subalgebra $\fn_0\subset\fg_0$ containing $\eta$ and $s_{\eta}\tau$, then relation \eqref{eqinBGG} holds.
\end{lem}
\begin{proof}
Let us first prove $(i)$.
Consider
the subalgebra spanned by $\widehat{\sigma}(e_{\eta})$ and $\widehat{\sigma}(e_{s_{\eta}\tau})$.
Then this subalgebra is isomorphic to $\fn_0$ and we obtain the needed conditions because
the module generated by $\widehat{\sigma}(e_{\eta})$ and $\widehat{\sigma}(e_{s_{\eta}\tau})$ from $v$ is
isomorphic to the (restriction of) simple module over $\fg_0$. This
completes the proof of $(i)$. The proof of $(ii)$ is completely analogous.
\end{proof}

\subsection{Type $A_2^{(2)}$}\label{A22}
In this subsection we consider the case of twisted affine algebra of type $A^{(2)}_2$.
So we start with $\fg=\msl_3$ and consider the twisted current algebra $\fg[t]^{(2)}$.
The roots system of $\fg_0=\mathfrak{so}_3$ is equal to $\Delta=\lbrace \alpha, -\alpha \rbrace$ and the real roots
of $\fg[t]^{(2)}$ are of the form $\pm\al+k\delta$, $k\ge 0$ and $\pm 2\al+ (2k+1)\delta$, $k\ge 0$.
The imaginary roots are of the form $k\delta$ and each weight space $\fh_{k\delta}$ is one-dimensional.
The tQBG contains two vertices labeled by ${\rm id}$ and $s$ and two arrows pointing in both directions.
The element $\widehat{s}$ acts as the transposition
of roots $\alpha$ and $-2\alpha+\delta$ and of root vectors $e_{\alpha}$ and $e_{-2\alpha+ \delta}$.

The sequence $\bar\beta^1$ is of length one, i.e.
$\bar\beta^1=(\beta)$ and $\beta= -\alpha+\delta.$
In particular, ${\rm deg} \beta=1$.
For an antidominant weight $\la=-m\om$ one has ${\rm deg}\beta^{1,\la}={\rm deg} \beta + \langle -\al,\-m\om\rangle =m$.
For $m \geq 0$ the defining relations of the module $W_{-m\omega}$ are
\begin{gather*}
h v=0,\ h\in \fh_{k\delta},\ k \geq 1,\\
e_{-\al+k\delta} v=0, e_{-2\al+(2k-1)\delta}v=0, k > 0,\ e_{\alpha}^{m+1}v=0,
\end{gather*}
and the defining relations of the module $W_{m\omega}=W_{s(-m\om)}$ are
\begin{gather*}
h v=0,\ h\in \fh_{k\delta},\ k \geq 1,\\
e_{\al+k\delta} v=0, e_{2\al+(2k+1)\delta}v=0, k \geq 0,\ e_{-2\alpha+\delta}^{\left[\frac{m}{2}\right]+1}v=0.
\end{gather*}
By definition, $W_{2k\omega}\simeq W_{(2k+1)\omega}$. This proves Theorem \ref{stepofdecomposition}, $(i)$, since
the edge from $s$ to ${\rm id}$ is quantum and Remark \ref{A2rec} tells us that if ${\rm deg}\beta^{1,\la}$ is odd
(i.e. $\la=(-2k-1)\om$ for some $k$), then $W_{s(\la)}\simeq W_{s(\la+\om)}$.

Let us prove Theorem \ref{stepofdecomposition}, $(ii)$.
Let $v$ be the generator of the module $W_{-m\omega}$, $m\ge 0$. Then direct computation shows that
\begin{gather*}
(e_{-2\alpha+ \delta}) ^{\frac{m}{2}}e_{\alpha}^mv=0,\ m \text{ is  even},\\
(e_{-2\alpha+ \delta}) ^{\frac{m+1}{2}}e_{\alpha}^mv=0,\ m \text{ is  odd}.
\end{gather*}

Now let $v$ be the generator of the module $W_{m\omega}$, $m>0$. Then
\[e_{\alpha}^{m}(e_{-2\alpha+ \delta})^{\left[\frac{m}{2}\right]}v=0.\]
This completes the proof of Theorem \ref{stepofdecomposition}, $(ii)$.

Here are the characters of the first several Weyl modules ($x=x^\om$):
\begin{gather*}
{\rm ch} W_0=1,\ {\rm ch} W_\om=x,\ {\rm ch} W_{-\om}=x^{-1}+x,\ {\rm ch} W_{2\om}=x^2+q+qx^2,\\
{\rm ch} W_{-2\om}=x^{-2}+1+x^2,\ {\rm ch} W_{3\om}=x{\rm ch} W_{2\om},\\
{\rm ch} W_{-3\om}=x^{-3}+x^{-1}(1+q)+x(1+q)+x^3,\\
{\rm ch} W_{4\om}=x^4+(q+q^2)x^2+q+q^2+q^3+(q^2+q^3)x^{-2}+q^2x^{-4},\\
{\rm ch} W_{-4\om}=x^{-4}+x^{-2}(1+2q)+(1+q)+x^2(1+q)+x^4,\\
{\rm ch} W_{5\om}=x{\rm ch} W_{4\om}.
\end{gather*}

The dimensions of the generalized Weyl modules are given by the formula
\[
\dim W_{m\om}=
\begin{cases}
3^{m/2}, \ m \text{ is even and nonnegative},\\
3^{-m/2},\ m  \text{ is even and negative},\\
3^{(m-1)/2}, \ m \text{ is odd and positive},\\
2\cdot 3^{(-m+1)/2},\ m \text{ is odd and negative}.\\
\end{cases}
\]

\subsection{Type $D_3^{(2)}$}
The goal of this section is to prove Theorem \ref{stepofdecomposition} for $\fg$ of type $D_3^{(2)}$.

We denote by $\alpha_1$ the short simple root, by $\alpha_2$ the long simple root,
$\Delta_+=\lbrace \alpha_1, \alpha_2, \alpha_2+\alpha_1, \alpha_2+2\alpha_1 \rbrace$ and the set of corresponding coroots is
$\lbrace \alpha_1^\vee, \alpha_2^\vee, 2\alpha_2^\vee+\alpha_1^\vee, \alpha_2^\vee+\alpha_1^\vee\rbrace$.
We have the following sequences of $\beta$'s:
\[\beta^1_1=-\alpha_1+\delta, \beta^1_2=-2\alpha_1-\alpha_2+2\delta, \beta^1_3=-\alpha_1-\alpha_2+\delta;\]
\[\beta^2_1=-\alpha_2+2\delta, \beta^2_2=-\alpha_1-\alpha_2+2\delta, \beta^2_3=-2\alpha_1-\alpha_2+2\delta, \beta^2_4=-\alpha_1-\alpha_2+\delta.\]

The twisted quantum Bruhat graph is shown on Figure \ref{QBGC2}.

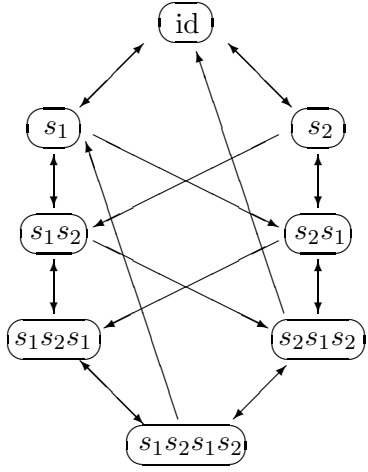
\begin{figure}\label{QBGC2}
\begin{picture}(200,200)

\put(100,180){\oval(20,15)}
\put(96,176){${\rm id}$}

\put(50,140){\oval(20,15)}
\put(46,138){$s_1$}

\put(150,140){\oval(20,15)}
\put(146,138){$s_2$}

\put(50,100){\oval(25,15)}
\put(41,98){$s_1s_2$}

\put(150,100){\oval(25,15)}
\put(141,98){$s_2s_1$}

\put(50,60){\oval(35,15)}
\put(35,58){$s_1s_2s_1$}

\put(150,60){\oval(35,15)}
\put(135,58){$s_2s_1s_2$}

\put(100,20){\oval(45,15)}
\put(82,18){$s_1s_2s_1s_2$}

\put(72,160){\vector(1,1){12}}
\put(72,160){\vector(-1,-1){12}}

\put(128,160){\vector(-1,1){12}}
\put(128,160){\vector(1,-1){12}}

\put(50,120){\vector(0,1){10}}
\put(50,120){\vector(0,-1){10}}

\put(150,120){\vector(0,1){10}}
\put(150,120){\vector(0,-1){10}}

\put(50,80){\vector(0,1){10}}
\put(50,80){\vector(0,-1){10}}

\put(150,80){\vector(0,1){10}}
\put(150,80){\vector(0,-1){10}}

\put(72,40){\vector(-1,1){12}}
\put(72,40){\vector(1,-1){12}}

\put(128,40){\vector(1,1){10}}
\put(128,40){\vector(-1,-1){10}}

\put(100,120){\vector(2,-1){35}}
\put(100,120){\line(-2,1){35}}

\put(100,80){\vector(2,-1){31}}
\put(100,80){\line(-2,1){35}}

\put(100,120){\vector(-2,-1){35}}
\put(100,120){\line(2,1){35}}

\put(100,80){\vector(-2,-1){31}}
\put(100,80){\line(2,1){35}}

\put(137,70){\vector(-1,3){33}}

\put(97,30){\vector(-1,3){35}}
\end{picture}
\caption{$tQBG$ of type $D_3^{(2)}$}
\end{figure}

\begin{prop}\label{edgesD32}
Let $\bar\beta^i$ be the sequence of $\beta$'s for some reduced decomposition of the element $t_{-\omega_i}$, $i=1,2$.
If there is no edge $\sigma \stackrel{{\rm Re} \beta_{m+1}}{\longrightarrow} \sigma s_{{\rm Re} \beta_{m+1}}$, then
\[W_{\sigma(\lambda)}(\bar\beta^i,m) \simeq W_{\sigma(\lambda)}(\bar\beta^i,m+1).\]
\end{prop}
\begin{proof}
Let $v$ be the cyclic vector of $W_{\sigma(\lambda)}(\bar\beta^i,m)$.
We need to prove the relation
\begin{equation}\label{desired}
(\widehat{\sigma}e_{- {\rm Re} \beta_{m+1}})^{l_{-{\rm Re} \beta_{m+1},m}}v=0
\end{equation}
assuming that there is no edge $\sigma \stackrel{{\rm Re} \beta_{m+1}}{\longrightarrow} \sigma s_{{\rm Re} \beta_{m+1}}$.

Lemma \ref{edgesinQBG} tells us that we need to consider two cases. Assume that
there are elements $\tau, \eta \in \Delta_+$ such that:
\[\tau, \eta \neq -{\rm Re} \beta_{m+1},\]
\[\tau + \eta=2\frac{\langle \tau,
{\rm Re} \beta_{m+1} \rangle}{\langle {\rm Re} \beta_{m+1}, {\rm Re} \beta_{m+1} \rangle}{\rm Re} \beta_{m+1},\]
\[\widehat{\sigma}(\tau) + \widehat{\sigma}(\eta)=
2\frac{\langle \tau, {\rm Re} \beta_{m+1} \rangle}{\langle {\rm Re} \beta_{m+1}, {\rm Re} \beta_{m+1} \rangle}\widehat{\sigma}
{\rm Re} \beta_{m+1}.\]
Then $-{\rm Re} \beta_{m+1}$ is equal to $\alpha_2+\alpha_1$ or  $\alpha_2+2\alpha_1$. Assume that $-{\rm Re} \beta_{m+1}=\alpha_2+\alpha_1$.
Then $\tau=\alpha_1$, $\eta=\alpha_2$ or $\tau=2\alpha_1+\alpha_2$, $\eta=\alpha_2$. We consider the first case, then
 $e_{\widehat{\sigma}(-{\rm Re}) \beta_{m+1}}$ is an element of the Lie algebra with simple root vectors
$e_{\widehat{\sigma}(\alpha_1)}$, $e_{\widehat{\sigma}(\alpha_2)}$. Using Corollary \ref{order}, $iii)$ we have that
$l_{{\rm Re} \beta_{m+1},m}>l_{\alpha_2,m}+2l_{\alpha_1,m}.$
Consider the Lie algebra spanned by vectors $e_{\widehat{\sigma}(\alpha_2)},e_{\widehat{\sigma}(\alpha_1)},
e_{\widehat{\sigma}(\alpha_1+\alpha_2)}, [e_{\widehat{\sigma}(\alpha_1+\alpha_2)}e_{\widehat{\sigma}(\alpha_1)}]$.
This subalgebra is isomorphic to the (finite-dimensional) Borel subalgebra of type $C_2$.
Thus using Remark \ref{BGG} we obtain that
$\widehat \sigma(e_{{\rm Re} \beta_{m+1}})^{l_{\alpha_2,m}+2l_{\alpha_1,m}+1}v=0$.
The second case ($\tau=2\alpha_1+\alpha_2$, $\eta=\alpha_2$) can be worked out in a similar way by the brute force computations.

Now assume that $-{\rm Re} \beta_{m+1}=\alpha_2+2\alpha_1$. Then $\tau=\alpha_1, \eta=\alpha_2+\alpha_1$.
Then either the linear span of
$\widehat\sigma e_{\alpha_1},\widehat\sigma e_{\alpha_2+2\alpha_1},\widehat\sigma e_{\alpha_2+\alpha_1},\widehat\sigma e_{\alpha_2}$
or the linear span of $
\widehat\sigma e_{\alpha_1},\widehat\sigma e_{\alpha_2+2\alpha_1},
\widehat\sigma e_{\alpha_2+\alpha_1},\widehat\sigma e_{-\alpha_2+ \delta}
$ is closed under the Lie bracket.
If the subspace spanned by
$\widehat\sigma e_{\alpha_1},\widehat\sigma e_{\alpha_2+2\alpha_1},\widehat\sigma e_{\alpha_2+\alpha_1},\widehat\sigma e_{\alpha_2}$
is closed under the Lie bracket then we can analogously to the previous case use Remark \ref{BGG}. Conversely, if
the subspace spanned by
\[
\widehat\sigma e_{\alpha_1},\widehat\sigma e_{\alpha_2+2\alpha_1},
\widehat\sigma e_{\alpha_2+\alpha_1},\widehat\sigma e_{-\alpha_2+ \delta}
\]
is closed under the Lie bracket, then the needed equation is equivalent to
\[
(\widehat\sigma e_{-\alpha_2+ \delta})^{l_{\alpha_2,m}+l_{\alpha_1,m}+1}
(\widehat\sigma e_{\alpha_2+\alpha_1})^{l_{\alpha_2+\alpha_1,m}+1}v=0.
\]

The remaining case is $\widehat\sigma(-{\rm Re} \beta_{m+1})\in \Delta_-$, $-{\rm Re} \beta_{m+1}=2\alpha_1 + \alpha_2$.
However in this case $\widehat\sigma (2\alpha_1 + \alpha_2)=\widehat\sigma (\alpha_1 + \alpha_2)+\widehat\sigma \alpha_1$
or $\widehat\sigma (2\alpha_1 + \alpha_2)=\widehat\sigma (\alpha_1 + \alpha_2)+\widehat\sigma \alpha_1+ \delta$.
If $\widehat\sigma (2\alpha_1 + \alpha_2)=\widehat\sigma (\alpha_1 + \alpha_2)+\widehat\sigma \alpha_1$, then
we are in the situation of the previous paragraph. Now assume that $\widehat\sigma (2\alpha_1 + \alpha_2)=
\widehat\sigma (\alpha_1 + \alpha_2)+\widehat\sigma \alpha_1+ \delta$. We know that exactly one of the elements
$\sigma(\al_1)$, $\sigma(\al_1+\al_2)$ is negative. We work out the case $\sigma(\al_1)\in\Delta_+$, the second case is very similar.
So assume that $\sigma(\al_1)\in\Delta_+$, $\sigma(\al_1+\al_2)\in\Delta_-$. Then the subalgebra of $\fn^{af}$
generated by the root vectors corresponding to roots
\[
\sigma(\al_2)+2\delta,\ \widehat{\sigma} (\al_1+\al_2) +\delta,\ \widehat{\sigma} (2\al_1+\al_2), \ \widehat{\sigma} (\al_1)
\]
is isomorphic to the nilpotent subalgebra of type $C_2$. Now the desired relation \eqref{desired} follows from the
Remark \ref{ndelta} and Remark \ref{BGG}.
\end{proof}

\begin{prop}\label{decompositionC_2}
Let $\bar\beta^i$ be the sequence of $\beta$'s for some reduced decomposition of the element $t_{-\omega_i}$, $i=1,2$.
Put $\eta=-{\rm Re} \beta_{m+1}$.
If there exists an edge $\sigma \stackrel{\eta}{\longrightarrow} \sigma s_{\eta}$,
then there exists a surjection
\[
W_{{\sigma}s_{\eta}(\lambda)}(\bar\beta^i,m+1) \twoheadrightarrow
{\rm U}(\fn^{af}) \widehat{\sigma}(e_{\eta})^{l_{\eta,m}} v.
\]
\end{prop}
\begin{proof}
We need to prove the following equalities:
\begin{equation}\label{reflectionequalityC2}
\widehat{\sigma s_{\eta}}(e_{\tau})^{l_{\tau,m+1}+1}
\widehat{\sigma}(e_{\eta})^{l_{\eta,m}} v =0.
\end{equation}

Note first that if both
$\eta$ and  $\tau$ are long roots then $\mathbb{Z}\langle \eta, \tau \rangle\simeq A_1 \oplus A_1$.
Therefore we can use Proposition \ref{subsystemofrank2}.

If $\eta$ and $s_{\eta}\tau$ are simple then we obtain the needed relation using Lemma \ref{simpleroots}.

Assume now that $\eta$ is a short simple root and $\tau$ is nonsimple short root.
Then the only cases not covered by Lemma \ref{simpleroots}, $ii)$ are $\sigma=s_{\alpha}$ and $\sigma=w_0$. In these
cases $[\widehat{\sigma}(e_{-\alpha_2+ \delta}),\widehat{\sigma}(e_{\alpha_1+\alpha_2})]=\widehat{\sigma}(e_{\alpha_1})$.
Applying operators $\widehat{\sigma}(e_{\alpha_1+ \alpha_2})\widehat{\sigma}(e_{-\alpha_2+ \delta})$ and
$\widehat{\sigma}(e_{\alpha_1})$ to the vector $\widehat{\sigma}(e_{\alpha_1+ \alpha_2})^{m_1+m_2+1}v$ we obtain
that the vectors $\widehat{\sigma}(e_{\alpha_1+ \alpha_2})^{m_1+m_2+1}\widehat{\sigma}(e_{\alpha_1})v$
and $\widehat{\sigma}(e_{\alpha_1+ \alpha_2})^{m_1+m_2}\widehat{\sigma}(e_{2\alpha_1+ \alpha_2})v$ are equal to zero.
Continuing such computations we obtain the needed relations. In the same way we obtain the needed relations
for the remaining cases, namely for $\eta$ and $\tau$ both simple roots.

Let $\eta$ be nonsimple. Then we obtain the needed relations by the weight reasons.
We work out one case in details, the other cases are very similar.

Let  $\eta={2\alpha_1+ \alpha_2}$ and
$\tau={\alpha_1+ \alpha_2}$. Then
\begin{equation}\label{wt}
wt(\widehat{\sigma}(e_{\tau})^{l_{\tau,m}+1}\widehat{\sigma}(e_{\eta})^{l_{\eta,m}}v)=
\sigma(\lambda)+l_{2 \alpha_1+\alpha_2,m}\sigma(2 \alpha_1+\alpha_2)+l_{\alpha_1+\alpha_2,m}\sigma(-\alpha_1).
\end{equation}
However, we know that there is no such weight space in $W_{\sigma(\la)}(\bar\beta^i,m)$. In fact,
the module $W_{\sigma(\la)}(\bar\beta^i,m)$ is generated by $\fn^{af}$ from $v$. The maximal number of times we can apply
$\widehat{\sigma}(\al_2)$ is equal to $l_{\al_2,m}$. Therefore all the weights of $W_{\sigma(\la)}(\bar\beta^i,m)$
are of the form $\sigma(\la)+k_1\sigma(\al_1)+k_2\sigma(\al_1+\al_2)$ (for some integer $k_1$ and $k_2$) with
$k_1\ge -l_{\al_2,m}$. But the right hand side of \eqref{wt} does not satisfy this condition, since
$l_{\al_1+\al_2,m}\ge l_{2\al_1+\al_2,m} + l_{\al_2,m}$.
\end{proof}

\subsection{Type $A^{(2)}_4$}
Consider now type $A^{(2)}_4$. Then the tQBG and the sequences $\bar\beta^1$, $\bar\beta^2$ are the same as in type $D^{(2)}_3$.
Let us fix an antidominant $\la$ and $i=1,2$ such that $\langle \la,\al_i\rangle<0$.
Assume that $\widehat{\sigma}(-{\rm Re} \beta^i_{m+1})$ is negative short and
${\rm deg}\beta^{i,\la+\om_i}_{m+1}=l$ is odd. Recall \eqref{ldeg} that $l=l_{-{\rm Re}\beta_{m+1}^i,m}$.
Therefore by definition the only difference between
$W_{\sigma(\lambda)}(\bar\beta^i, m+1)$ and $W_{\sigma(\lambda)}(\bar\beta^i, m)$ is the relation
$(\widehat{\sigma}e_{-{\rm Re} \beta^i_{m+1}})^{\left[\frac{l-1}{2}\right]+1}v=0$ instead of
$(\widehat{\sigma}e_{-{\rm Re} \beta^i_{m+1}})^{\left[\frac{l}{2}\right]+1}v=0$.
However for odd $l$ these relations coincide. Thus if $\widehat{\sigma}(-{\rm Re} \beta^i_{m+1})$ is negative short and
${\rm deg}\beta^{i,\la+\om_i}_{m+1}=l$ is odd, then $W_{\sigma(\lambda)}(\bar\beta^i, m+1)\simeq W_{\sigma(\lambda)}(\bar\beta^i, m)$.
The proof of all other conditions from Theorem \ref{stepofdecomposition}, $(i)$,
is straightforward.

\begin{prop}
Theorem \ref{stepofdecomposition}, $(ii)$ holds for twisted current algebra $A^{(2)}_4$.
\end{prop}
\begin{proof}
If $-{\rm Re} \beta^i_{m+1}$ is nonsimple then analogously to type $D^{(2)}_3$ we obtain the needed result
by the weight reasons. Thus we only need to consider the case $m=0$. Put $\eta=-{\rm Re} \beta^i_{m+1}$.
If $\eta=\alpha_2$ and $\tau=\alpha_2+2\alpha_1$, then we can use Lemma \ref{subsystemofrank2} to complete the proof.
Note that the vectors $\widehat{\sigma}(e_{\alpha_1})$ and $\widehat{\sigma}(e_{\alpha_2})$ generate the Lie algebra isomorphic
to the nilpotent subalgebra of type $C_2$ (the short and the long roots can be interchanged). Thus if $s_{\eta}(\tau)$ is simple,
then one can use Remark \ref{BGG}. All remaining cases can be obtained by a direct computation.
\end{proof}

\subsection{Type $D_4^{(3)}$}
Let $\fg^{af}$ be the affine Lie algebra of type $D_4^{(3)}$. Then the
subalgebra $\fg_0$ is of type $G_2$.
Denote by $\alpha_1$ the short simple root and by $\alpha_2$ the long simple root of $\fg_0$.
The tQBG of type $G_2$ can be obtained from QBG of type $G_2$ (see e.g. \cite{LL}, p.19, figure $2$) by interchanging $s_1$ and $s_2$
in the labels of all vertices and keeping all the arrows unchanged.
In words, the twisted quantum Bruhat graph looks as follows:
there are Bruhat edges from any element of length $p$ to any element of length $p+1$, $0 \leq p \leq 5$. There is quantum
edge from any element with the reduced decomposition $(\prod s_{i_k})s_j$ to the element $(\prod s_{i_k})$, $j, i_k \in \lbrace 1,2 \rbrace$,
from any element with the reduced decomposition $(\prod s_{i_k})s_2 s_1 s_2$ to the element $(\prod s_{i_k})$ and
from any element with the reduced decomposition $(\prod s_{i_k})s_1s_2 s_1 s_2s_1$ to the element $(\prod s_{i_k})$.

Using Proposition \ref{descriptionbeta} we obtain the following sequences $\bar \beta^1, \bar \beta^2$:
\begin{multline}\label{betasG21}
{\beta_1^1}=-\alpha_1+ \delta, {\beta_2^1}=-3\alpha_1-\alpha_2+3\delta,{\beta_3^1}=-2\alpha_1-\alpha_2+2 \delta,\\
{\beta_4^1}=-3\alpha_1-2\alpha_2+3\delta,{\beta_5^1}=-\alpha_2-\alpha_1+\delta,
{\beta_6^1}=-2\alpha_1-\alpha_2+ \delta.
\end{multline}

\begin{multline}\label{betasG22}
{\beta_1^2}=-\alpha_2+ 3\delta,{\beta_2^2}=-\alpha_2-\alpha_1+3 \delta, {\beta_3^2}=-3\alpha_1-2\alpha_2+6\delta,
{\beta_4^2}=-\alpha_2-2\alpha_1+3 \delta,\\
{\beta_5^2}=-\alpha_2-\alpha_1+2 \delta, {\beta_6^2}=-3\alpha_1-\alpha_2+3\delta,
{\beta_7^2}=-2\alpha_1-\alpha_2+2 \delta,\\
{\beta_8^2}=-3\alpha_1-2\alpha_2+3\delta, {\beta_9^2}=-\alpha_2-\alpha_1+\delta,
{\beta_{10}^2}=-2\alpha_1-\alpha_2+\delta.
\end{multline}

\begin{prop}
Assume that there is no edge $w \stackrel{\alpha}{\longrightarrow} w s_{{{\rm Re} \beta^i_{m+1}}}$. Then
\[W_{\sigma(\lambda)}(\bar\beta^i,m) \simeq W_{\sigma(\lambda)}(\bar\beta^i,m+1).\]
\end{prop}
\begin{proof}
Lemma \ref{edgesinQBG} tells us that we need to consider two cases. Assume that
there exist elements $\tau, \eta \in \Delta_+$ such that
\begin{gather*}
\tau, \eta \neq -{\rm Re} \beta^i_{m+1},\\
\tau + \eta=2\frac{\langle \tau,
{\rm Re} \beta^i_{m+1} \rangle}{\langle {\rm Re} \beta^i_{m+1}, {\rm Re} \beta^i_{m+1} \rangle}{\rm Re} \beta^i_{m+1},\\
\widehat{\sigma}(\tau) + \widehat{\sigma}(\eta)=
2\frac{\langle \tau, {\rm Re} \beta^i_{m+1} \rangle}{\langle {\rm Re} \beta^i_{m+1}, {\rm Re} \beta^i_{m+1} \rangle}\widehat{\sigma}
{\rm Re} \beta^i_{m+1}.
\end{gather*}
Then we have two cases:
\begin{gather*}
\tau + \eta=-{\rm Re} \beta^i_{m+1},\\
\widehat{\sigma}(\tau) + \widehat{\sigma}(\eta)= -\widehat{\sigma} {\rm Re} \beta^i_{m+1}
\end{gather*}
or
\begin{gather*}
\tau + \eta=-3{\rm Re} \beta^i_{m+1},\\
\widehat{\sigma}(\tau) + \widehat{\sigma}(\eta)= -3\widehat{\sigma} {\rm Re} \beta^i_{m+1}.
\end{gather*}
Let us work out the first case (in the second case $\beta^i_{m+1}$ is short nonsimple and this case can be worked out by the brute force computation).
Assume that $-{\rm Re} \beta^i_{m+1}=\alpha_1+\alpha_2$. Then $m>0$ and
$l_{-{\rm Re}\beta^i_{m+1},m}>3l_{\alpha_2,m}+l_{\alpha_1,m}$. But using Remark \ref{BGG}
we have:
\[(\widehat{\sigma}(e_{-{\rm Re} \beta^i_{m+1}}))^{3l_{\alpha_1,m}+l_{\alpha_2,m}+1}v=0.\]

Now assume that $-{\rm Re} \beta^i_{m+1}=2\alpha_1+\alpha_2$. Then $-{\rm Re}l_{\beta^i_{m+1},m}>l_{\alpha_1+\alpha_2,m}+l_{\alpha_2,m}$.
 If the set $[\widehat{\sigma}(f_{\alpha_2}),\widehat{\sigma}(f_{\alpha_1})]=\widehat{\sigma}(f_{\alpha_1+\alpha_2})$,
 then using Remark \ref{BGG} we have:
\[(\widehat{\sigma}(e_{-{\rm Re} \beta^i_{m+1}}))^{l_{\alpha_1+\alpha_2,m}+l_{\alpha_2,m}}v=0.\]

Conversely we have that $[(e_{\widehat\sigma(-\alpha_1+\delta)}),\widehat\sigma(e_{\alpha_1+\alpha_2})]=\widehat\sigma(e_{\alpha_2})$
and using this fact we obtain:
\[\widehat{\sigma}(e_{-{\rm Re} \beta^i_{m+1}}))^{l_{\alpha_1+\alpha_2,m}+l_{\alpha_2,m}+1}v=0.\]

In the similar way we prove the claim for  $-{\rm Re} \beta^i_{m+1}=\alpha_1+3\alpha_2$ or $-{\rm Re} \beta^i_{m+1}=2\alpha_1+3\alpha_2$.

Now assume that there do not exist such $\tau$ and $\eta$ that
$-{\rm Re}\beta^i_{m+1}$ is nonsimple long and $\widehat \sigma (-{\rm Re}\beta^i_{m+1}) \in \Delta_-$.
Then the only possible cases are $\sigma=w_0$ or $\sigma=s_{2 \alpha_1+3 \alpha_2}$. Then using the direct computation we obtain
that $(\widehat \sigma e_{-{\rm Re}\beta^i_{m+1}})^{l_{-{\rm Re}\beta^i_{m+1},m}}$ lie in the left ideal generated by
$(\widehat \sigma e_{\alpha})^{l_{\alpha,m}}$,
$\alpha\neq -{\rm Re}\beta^i_{m+1}$.
\end{proof}

\begin{prop}
We consider a module $W_{\sigma(\lambda)}(\bar\beta^i,m)$. If there exists an edge
\[\sigma \stackrel{{\rm Re}\beta^i_{m+1}}{\longrightarrow} \sigma s_{{{\rm Re} \beta^i_{m+1}}}\]
in the quantum Bruhat graph, then $U(\mathfrak{n}^{af}) \widehat \sigma(e_{{-{\rm Re} \beta^i_{m+1}}})^{l_{-{\rm Re} \beta^i_{m+1},m}}v$ is the quotient module of
$W_{\sigma s_{{{\rm Re} \beta^i_{m+1}}}(\lambda)}$.
\end{prop}

\begin{proof}
Let $v_1=\hat \sigma(e_{{-{\rm Re} \beta^i_{m+1}}})^{l_{-{\rm Re} \beta^i_{m+1},m}}v$.
If $\langle {\rm Re} \beta^i_{m+1},\eta \rangle=0$, then it is easy to see that $[e_{-{\rm Re} \beta^i_{m+1}},e_\eta]=0$ and
thus $\mathbb{Z}\langle {\rm Re} \beta^i_{m+1},\tau \rangle\cap \Delta$ is the root system of type $A_1\oplus A_1$. Therefore
the claim is a consequence of the Lemma \ref{subsystemofrank2}.

If ${\rm Re} \beta^i_{m+1}$ is long, then for any long root $\tau \neq -{\rm Re} \beta^i_{m+1}$ we have that
$\mathbb{Z}\langle {\rm Re} \beta^i_{m+1},\tau \rangle\cap \Delta$ is a root system of type $A_2$.
Indeed, the Lie algebra spanned by all long roots of $G_2$ is isomorphic to $A_2$.
Hence the claim is a consequence of the Lemma \ref{subsystemofrank2}.

Consider the following cases:
$1)-{\rm Re} \beta^i_{m+1}=\alpha_2+2\alpha_1$, $2)-{\rm Re} \beta^i_{m+1}=2\alpha_2+3\alpha_1$,
$3)-{\rm Re} \beta^i_{m+1}=\alpha_2+\alpha_1, \tau \neq \alpha_1$, $4)-{\rm Re} \beta^i_{m+1}=\alpha_2+3\alpha_1, \tau \neq \alpha_2$. Then
the needed condition holds because of the weight reasons in the same way as in the proof for $D^{(2)}_3$.

Now assume that $-{\rm Re} \beta^i_{m+1}=\alpha_2$. Then $m=0$. Note that the cases of long $\tau$ or $\tau$ orthogonal to
$\alpha_1$  are already
covered. The case $\tau=\alpha_1+\alpha_2$ is covered by Remark \ref{BGG}. $2\alpha_1+\alpha_2$ is orthogonal to $\alpha_2$. Thus it
remains just to consider case $\tau=\alpha_1$. We need to prove the following equation:
\[\widehat{\sigma}(e_{\alpha_1+\alpha_2})^{m_1+1}\widehat{\sigma}(e_{\alpha_2})^{m_2}v=0.\]
This and all remaining cases can be proven by a direct computation.
\end{proof}

\section*{Acknowledgments}
We are grateful to Ghislain Fourier and Daniel Orr for useful remarks and discussions. 
The work was supported by the RSF-DFG grant 16-41-01013.


\begin{thebibliography}{99}



\bibitem[BFP]{BFP} F.~Brenti, S.~Fomin, A.~Postnikov, {\it Mixed Bruhat operators
and Yang-Baxter equations for Weyl groups}, Internat. Math. Res. Notices 1999, no. 8, 419--441.

\bibitem[C]{C}
V.~Chari, {\it On the fermionic formula and the Kirillov-Reshetikhin conjecture},
Internat. Math. Res. Notices 2001, no. 12, 629--654.

\bibitem[CFS]{CFS} Vyjayanthi Chari, Ghislain Fourier, and Prasad Senesi. {\it Weyl modules for the twisted loop algebras,} J. Algebra,
319(12):5016–5038, 2008.

\bibitem[CI]{CI}
V.~Chari, B.~Ion, {\it BGG reciprocity for current algebras}, Compositio Mathematica 151 (2015), pp. 1265--1287.

\bibitem[CIK]{CIK}
V.~Chari, B.~Ion, D.~Kus, {\it Weyl modules for the hyperspecial current algebra},
Int. Math. Res. Not. IMRN 2015, no. 15, 6470--6515.

\bibitem [CL]{CL}
{V.~Chari}, {S.~Loktev},
{\it Weyl, Demazure and fusion modules for
the current algebra of $\msl_{r+1}$}, Adv. Math. 207 (2006), 928--960.

\bibitem[CP]{CP}
V.~Chari and A.~Pressley, {\it Weyl modules for classical and quantum affine algebras}, Represent.
Theory, 5:191--223 (electronic), 2001.


\bibitem [Ch1] {Ch1}
{I.~Cherednik},
{\it Nonsymmetric Macdonald polynomials},
IMRN {10} (1995), 483--515.

\bibitem [Ch2] {Ch2}
I.~Cherednik,
{\it Double affine Hecke algebras},
London Mathematical Society Lecture
Note Series, {319}, Cambridge University Press, Cambridge, 2006.

\bibitem[Ch3]{Ch3}
I.~Cherednik, {\it DAHA-Jones polynomials of torus knots}, arxiv.org/abs/1406.3959.

\bibitem [CF] {CF}
I.~Cherednik, E.~Feigin, {\it Extremal part of the PBW-filtration and E-polynomials},
Advances in Mathematics (2015), vol. 282, pp. 220--264.


\bibitem [CO1] {CO1}
I.~Cherednik, {D.~Orr},
{\it Nonsymmetric difference Whittaker functions}, Math. Z. 279 (2015), no. 3--4, 879--938.

\bibitem [CO2] {CO2}
I.~Cherednik, {D.~Orr},
{\it One-dimensional nil-DAHA and Whittaker functions},
Transformation Groups { 18}:1 (2013), 23--59.

\bibitem[FK]{FK}
G.~Fourier, D.~Kus, {\it Demazure and Weyl modules: The twisted current case},
Transactions of the AMS, Volume 365, p. 6037--6064 (2013).

\bibitem[FeLo]{FeLo} B.~Feigin, S.~Loktev,
{\it On generalized Kostka polynomials and the quantum Verlinde rule},
Differential topology, infinite-dimensional Lie algebras, and applications, 61--79, Amer. Math. Soc. Transl. Ser. 2, 194,
Amer. Math. Soc., Providence, RI, 1999.

\bibitem[FL1]{FL1} G.~Fourier, P.~Littelmann,
{\it Tensor product structure of affine Demazure modules and limit constructions},
Nagoya Math. Journal 182 (2006), 171--198.


\bibitem[FL2]{FL2} G.~Fourier, P.~Littelmann,
{\it Weyl modules, Demazure modules, KR-modules, crystals, fusion products and limit constructions},
Advances in Mathematics 211 (2007), no. 2, 566--593.

\bibitem[FM1]{FM1}
E.~Feigin, I.~Makedonskyi,
{\it Nonsymmetric Macdonald polynomials and PBW filtration: towards the proof of the Cherednik-Orr conjecture},
Journal of Combinatorial Theory, Series A (2015), pp. 60--84.

\bibitem[FM2]{FM2}
E.~Feigin, I.~Makedonskyi, {\it Weyl modules for $osp(1,2)$ and nonsymmetric Macdonald polynomials},
arxiv.1507.01362, to appear in Mathematical Research Letters.

\bibitem[FM3]{FM3}
E.~Feigin, I.~Makedonskyi, {\it Generalized Weyl modules, alcove paths and Macdonald polynomials},
arXiv:1512.03254.

\bibitem[FMO]{FMO}
E.~Feigin, I.~Makedonskyi, D.~Orr, {\it Generalized Weyl modules and nonsymmetric $q$-Whittaker functions},
arXiv:1605.01560.

\bibitem[GL]{GL} S.~Gaussent and P.~Littelmann, {\it LS galleries, the path model, and MV cycles},
Duke Math.J. 127 (2005), no. 1, 35--88.

\bibitem[H]{H}
{M.~Haiman}, {\it Cherednik algebras, Macdonald polynomials and combinatorics},
Proceedings of the International Congress of Mathematicians, Madrid 2006, Vol. III, 843--872.

\bibitem [HHL] {HHL}
{M.~Haiman}, and {J.~Haglund}, and {N.~Loehr},
{\it A combinatorial formula for non-symmetric Macdonald polynomials},
Amer. J. Math. { 130}:2 (2008), 359--383.

\bibitem[HKOTY]{HKOTY}
G.~Hatayama, A.~Kuniba, M.~Okado, T.~Takagi, Y.~Yamada,
{\it Remarks on fermionic formula}, Recent developments in quantum affine algebras and related topics (Raleigh, NC, 1998), 243--291,
Contemp. Math., 248, Amer. Math. Soc., Providence, RI, 1999. http://arxiv.org/pdf/math/9812022v3.pdf


\bibitem[I]{I} {B.~Ion},
{\em Nonsymmetric Macdonald polynomials and Demazure characters},
Duke Mathematical Journal {116}:2 (2003), 299--318.

\bibitem[J]{J} A.~Joseph, {\it On the Demazure character formula}, Annales Scientifique de l’.E.N.S.,
1985, 389--419.

\bibitem[Kac]{Kac}
V.~Kac, Infinite dimensional Lie algebras, 3rd ed.,
Cambridge University Press, Cambridge, 1990.

\bibitem[Kn]{Kn}
F.~Knop, {\it Integrality of two variable Kostka functions},
Journal fuer die reine und angewandte Mathematik 482 (1997) 177--189.

\bibitem[Kum]{Kum}
S.~Kumar, {\it Kac-Moody groups, their flag varieties and representation theory},
Progress in Mathematics, 204. Birkh{\" a}user Boston, Inc., Boston, MA, 2002.


\bibitem[KS]{KS}
F.~Knop, S.~Sahi, {\it A recursion and a combinatorial formula for Jack polynomials},
Invent. Math. 128 (1997), no. 1, 9--22.


\bibitem [L] {L}
{C.~Lenart},
{\it From Macdonald polynomials to a charge statistic beyond type A}, Journal of Combinatorial Theory, Series A,
vol. 119 (3),2012, pp. 683--712.

\bibitem[Lu]{Lu}
G.~Lusztig.
Hecke algebras and Jantzen's generic decomposition patterns.
Advances in Math. 37 (1980), 121--164.

\bibitem [LL] {LL}
C.~Lenart, A.~Lubovsky, {\it A uniform realization of the combinatorial $R$-matrix}, http://arxiv.org/abs/1503.01765.

\bibitem [LS] {LS}
C.~Lenart and A.~Schilling, {\it Crystal energy functions via the charge in types A and C},
Math.Z.273 (2013), no. 1-2, 401--426.

\bibitem[LNSSS1]{LNSSS1} C.~Lenart, S.~Naito, D.~Sagaki, A.~Schilling, and M.~Shimozono, {\it Quantum Lakshmibai-Seshadri paths and
root operators, preprint 2013}, arXiv:1308.3529.

\bibitem[LNSSS2]{LNSSS2} C.~Lenart, S.~Naito, D.~Sagaki, A.~Schilling, and M.~Shimozono,
{\it A uniform model for Kirillov-Reshetikhin crystals I: Lifting the parabolic quantum Bruhat graph},
Int. Math. Res. Not. 2015 (2015), 1848--1901.

\bibitem[LNSSS3]{LNSSS3} C.~ Lenart, S.~Naito, D.~Sagaki, A.~Schilling, and M.~Shimozono,
{\it A uniform model for Kirillov-Reshetikhin
crystals II: Alcove model, path model, and $P = X$}, preprint 2014, arXiv:1402.2203.


\bibitem [LNSSS4] {LNSSS4}
C.~Lenart, S.~Naito, D.~Sagaki, A.~Schilling, M.~Shimozono,
{\it A uniform model for Kirillov-Reshetikhin crystals III: Nonsymmetric Macdonald polynomials at $t=0$ and Demazure characters},
arXiv:1511.00465.

\bibitem[M1]{M1}
I.G.~Macdonald, {\it Symmetric functions and Hall polynomials}, second ed., Oxford University Press, 1995.

\bibitem[M2]{M2}
I.G.~Macdonald, {\it Affine Hecke algebras and orthogonal polynomials}, S{\' e}minaire Bourbaki, Vol. 1994/95.
Ast{\' e}risque No. 237 (1996), Exp. No. 797, 4, 189--207.

\bibitem[N]{N}
K.~Naoi, {\it Weyl modules, Demazure modules and finite crystals for non-simply laced type},
Adv. Math. 229 (2012), no. 2, 875--934.

\bibitem[NNS]{NNS}
S.~Naito, F.~Nomoto, and D.~Sagaki.
An explicit formula for the specialization of nonsymmetric Macdonald polynomials at $t=\infty$.
arXiv:1511.07005

\bibitem[NS]{NS}
S.~Naito and D.~Sagaki.
Demazure submodules of level-zero extremal weight modules and specializations of Macdonald polynomials.
arXiv:1404.2436

\bibitem[O]{O}
E.~Opdam {\it Harmonic analysis for certain representations of graded Hecke algebras}, Acta Math. 175 (1995), no. 1, 75--121.

\bibitem[OS]{OS}
D.~Orr, M.~Shimozono, {\it Specializations of nonsymmetric Macdonald-Koornwinder polynomials},
arXiv:1310.0279.

\bibitem[RY]{RY} {A.~Ram, M.~Yip},
{A combinatorial formula for Macdonald polynomials}, Advances in Mathematics, vol. 226 (1), 2011, pp. 309--331.

\bibitem[Sage]{Sage} SageMath, {Nonsymmetric Macdonald polynomials package} by A.~Schilling and N.~M.~Thiery (2013),
http://doc.sagemath.org/html/en/reference/combinat/sage/\-combinat/\-root\_system/\-non\_symmetric\_macdonald\_polynomials.html.

\bibitem[S]{S} {Y.~Sanderson},
{\em On the Connection Between Macdonald Polynomials and
Demazure Characters},
J. of Algebraic Combinatorics, {11} (2000), 269--275.
\end{thebibliography}
\end{document}